\documentclass[final,10pt,twocolumn]{elsarticle}
\let\today\relax
\makeatletter
\def\ps@pprintTitle{%
    \let\@oddhead\@empty
    \let\@evenhead\@empty
    \def\@oddfoot{\footnotesize\itshape
         {} \hfill\today}%
    \let\@evenfoot\@oddfoot
    }
\makeatother

\usepackage{amsmath}
\usepackage{amssymb}
\usepackage{amsfonts}
\usepackage{amsthm}
\usepackage{mathptmx} 
\usepackage{times}
\usepackage{graphicx}
\usepackage{color}
\usepackage{hyperref}

\newtheorem{theorem}{Theorem}

\def\supp{\mathrm{supp}}
\def\sn{{| \! | \! |}}
\def\<{\leqslant}           
\def\>{\geqslant}           

\def\d{\partial}
\def\wh{\widehat}

\def\Re{{\rm Re}}   

\def\cH{{\cal H}}   
\def\mZ{\mathbb{Z}}    
\def\mR{\mathbb{R}}    
\def\mC{\mathbb{C}}    
\def\mH{\mathbb{H}}    
\def\Tr{{\rm Tr}}       
\def\rT{{\rm T}}        
\def\rF{\mathrm{F}}        

\def\bE{{\bf E}}    

\def\[[[{[\![\![}
\def\]]]{]\!]\!]}

\def\bra{{\langle}}
\def\ket{{\rangle}}

\def\re{{\rm e}}        
\def\rd{{\rm d}}        

\def\bA{\mathbf{A}}

\def\x{\times}

\def\fS{\mathfrak{S}}

\def\bD{{\mathbf D}}

\def\cF{{\cal F}}

\def\cov{{\bf cov}}

\def\cN{{\cal N}}
\def\cS{{\mathcal S}}
\def\cT{\mathcal{T}}

\def\mT{\mathbb{T}}

\def\eps{\epsilon}
\def\Ups{\Upsilon}
\def\ups{\upsilon}

\begin{document}

\begin{frontmatter}


\title{{\bf 
Anisotropy-based Robust Performance Criteria for Statistically Uncertain Linear Continuous Time  Invariant Stochastic Systems}}

\author[address]{Igor G. Vladimirov}
\address[address]{Australian National University,
        Canberra, Australia,
        E-mail:
            {\tt igor.g.vladimirov@gmail.com}
}
\begin{abstract}
This paper is concerned with robust performance criteria for linear continuous time invariant stochastic systems  driven by statistically uncertain random processes. The uncertainty is understood as the deviation of imprecisely known probability distributions of the input disturbance from those of the standard Wiener process.
Using a one-parameter family of conformal maps of the unit disk in the complex plane onto the right half-plane for discrete and continuous time transfer functions, the deviation from the nominal Gaussian white-noise model is quantified by the mean anisotropy for the input of a discrete-time  counterpart of the original system. The parameter of this conformal correspondence specifies the time scale for filtered versions of the input and output of the system, in terms of which
the worst-case root mean square gain is formulated subject to an upper constraint on the mean anisotropy.
The resulting two-parameter counterpart of the anisotropy-constrained norm of the system for the continuous time case is amenable to state-space computation using the methods of the anisotropy-based theory of stochastic robust filtering and control, originated by the author in the mid 1990s.
\end{abstract}



\end{frontmatter}
\footnotetext{}

\section{Introduction}

The idea of quantifying the ``amplification'' (or ``attenuation'') properties of a linear operator from one normed space to another in terms of its induced norm is ubiquitous in functional analysis, matrix theory and numerical methods, to mention some of the relevant areas. 
Using such a norm for operators on Hilbert spaces underlies the $\cH_\infty$-control theory \cite{F_1987} for linear systems with square integrable inputs.
When it is preferable to have low sensitivity of the output variables of the system to the input disturbances, the corresponding performance criteria are concerned with stabilizing the closed-loop system (thus making it a bounded input-output operator) and minimizing its operator norm by an appropriate choice of a controller \cite{DGKF_1989}. This includes not only control settings as such, but also filtering problems, where the role of the output process is played by the state estimation error \cite{NK_1991}.
Due to submultiplicativity of the operator norm,  its minimization leads not only to disturbance attenuation with respect to the input but also with respect to perturbations in the system itself. The small-gain theorem \cite{Z_1981} models such perturbations as feedback loops with norm-bounded uncertainties and is closely related to the invertibility  of a perturbed identity operator and von Neumann series in Banach algebras \cite{DS_1958}.

The Rayleigh quotient \cite{HJ_2007}, whose maximization defines the induced norm of a bounded operator on a Hilbert space (or a matrix in a finite-dimensional case), reduces to a quadratic form on the unit sphere. In a generic case, its  maximum over the sphere is achieved at two points, which specify the worst-case direction for the input. This direction is exceptional and is not necessarily taken by the unit vector  of the normalised input disturbance. Therefore, the operator norm is a conservative measure of the gain in the case when the disturbance is not targeted at specific directions. At its extreme, this suggests the averaging of the Rayleigh quotient over the uniform probability distribution on the unit sphere, which leads to an appropriately rescaled Frobenius norm of the operator. This approach is applicable only to a finite-dimensional Hilbert space, where the unit sphere has a (unique up to a multiplicative  constant) finite Haar measure \cite{A_1963}, invariant under the group of rotations, whose normalization yields the uniform distribution (there is no such distribution on an infinite-dimensional sphere, which is closely related to its noncompactness).

The deviation of an arbitrary probability measure on the sphere from the uniform distribution (as a reference measure) can be quantified in terms of entropy-theoretic proximity criteria. For example, the Kullback-Leibler relative entropy \cite{CT_2006} (with respect to the uniform distribution on the sphere) leads to the anisotropy functional \cite{V_1995a}. If the input disturbance is random and its direction distribution is ``nearly'' uniform in the sense that its anisotropy does not exceed a given nonnegative level $a$, then the maximization of the averaged Rayleigh quotient over such distributions leads to the $a$-anisotropic norm of the matrix. This norm occupies an intermediate position between the rescaled Frobenius norm, mentioned above, and the operator norm. Moreover, the latter two norms are the extreme cases of the $a$-anisotropic norm at $a=0$ and as $a\to +\infty$, respectively. The advantage of this norm is that its conservativeness is controlled by the parameter $a$ which specifies the amount of statistical uncertainty in the direction distribution (with the uniform distribution playing the role of a ``centre'' of the uncertainty class). An equivalent interpretation is that $a$ quantifies how far a hypothetical opponent can go in approximating the worst-case direction (corresponding to the largest singular value of the matrix) by an absolutely continuous probability distribution on the unit sphere. Although the anisotropy functional does not lend itself to closed-form calculation even for the direction distribution  of a Gaussian random vector, its asymptotic growth rate (per time step) is computable for unboundedly growing fragments of stationary Gaussian random sequences, leading to the mean anisotropy \cite{V_1995a}.

Later, the ``spherical'' anisotropy functional 
was complemented with  its more tractable counterpart \cite{V_2006}  (on the space of inputs themselves  rather than the unit sphere of their directions) in the form of the minimum relative entropy of the actual probability distribution with respect to isotropic Gaussian distributions,
which coincides with a multivariable  version of a power-entropy construct considered in a different context and for different purposes in \cite{B_1998}. Although this second anisotropy functional is an upper bound for the original one, it has  the same infinite-horizon growth rate in the stationary Gaussian case mentioned above, with the resulting mean anisotropy being expressed in terms of the spectral density and the mean value of the sequence. 
Accordingly, the averaged Rayleigh quotient was replaced with the ratio of root-mean-square (RMS) values of the output and input.

The anisotropy functionals, the anisotropy-constrained norms and the anisotropy-based theory of stochastic robust control and filtering (for systems without internal perturbations)   were originated by the author  in a series of papers and research reports in the mid 1990s -- early 2000s, including
\cite{V_1995a,V_1994,V_1995b,V_1999,V_1996a,V_1996b,V_1997,V_2001a,V_2001b,V_2001c,V_2003,V_2006,V_2008} and a set of MATLAB functions for anisotropy-based robust performance analysis and controller design. Being
motivated as a stochastic extension of the $\cH_\infty$-control theory to statistically uncertain linear discrete-time systems (and Toeplitz operators acting on homogeneous Gaussian random fields on multidimensional lattices), this development aimed to bridge the gap  between the deterministic approach of $\cH_\infty$-control and the stochastic paradigm of linear quadratic Gaussian (LQG) control, which includes Kalman filtering \cite{K_1960} and its predecessor --- the Kolmogorov-Wiener-Hopf theory of smoothing, filtering and prediction  for stationary random processes \cite{K_1941,W_1949}. This extension also addressed the robustness issues of the LQG approach which is oriented to an idealised scenario, when the random input disturbances have precisely known statistical characteristics and are organised as a Gaussian white-noise sequence or a standard Wiener process  (from which a more complicated covariance structure can be  obtained by using a shaping filter).

Not discussing here the other ways of combining the $\cH_\infty$-theory and LQG approaches (see, for example, \cite{BH_1989}), we note  that
the anisotropy-based theory offers probabilistic, system theoretic and computational tools (including a homotopy method for solving specific sets of cross-coupled Riccati, Lyapunov and log-determinant matrix algebraic equations) for robust performance analysis and synthesis of controllers and filters. This approach  addresses robustness with respect to spatially  and temporally coloured random disturbances\footnote{in the sense of statistical correlations between the entries of a multivariable noise at the same or different moments of time} with imprecisely known statistical properties, with the statistical uncertainty (as a deviation from the nominal Gaussian white noise model) being quantified in terms of anisotropy, and the anisotropic norm describing the worst-case RMS gain of the system over this uncertainty class.

Subsequently (in the 2000s and more recently), the theory  was being developed in the form of a suboptimal guaranteed approach to 
systems with internal perturbations \cite{KM_2004}, towards convexification of the anisotropy-based control synthesis \cite{MKV_2011,TKT_2011,T_2012b} and suboptimal observer design \cite{T_2013,TTKK_2018}, and also taking into account  nonzero-mean input disturbances  \cite{KT_2017} and multiplicative noise \cite{K_2018} in the system, to mention some of the developments.
The results of the anisotropy-based theory have been adapted to descriptor systems in \cite{BAK_2018}\footnote{where there are mathematical and bibliographic errors and inaccuracies:  for example, \cite[Definition 3.1 on p. 61]{BAK_2018} introduces erroneous dimensions of vectors and corresponding spaces in the definition of anisotropy and a wrong  sign of the differential entropy; \cite[p. 62]{BAK_2018} specifies a wrong analyticity domain for transfer functions; the order of noncommuting matrix factors in the factorization of spectral densities in \cite[Eq. (3.5) on p. 62]{BAK_2018}  is incorrect, and the same error is present in the proof of \cite[Theorem 4.1 on pp. 103--105]{BAK_2018}; there is an incorrect interpretation of some of the sets participating in \cite[Lemma 4.1 on p. 100]{BAK_2018}, which incorrectly mentions a ``saddle point of a set-valued map''   whereas the lemma is concerned with a fixed point of a set-valued map, describing a saddle point of the minimax problem; the first work on the anisotropy-based theory was \cite{V_1995a}, which was written by the author in late 1993 -- early 1994, presented soon afterwards  to M.S.Pinsker, A.Yu.Veretennikov and M.L.Kleptsyna at a stochastic analysis seminar  in the Institute for Information Transmission Problems, the Russian Academy of Sciences, and communicated to the Doklady Mathematics by Ya.Z.Tsypkin on 3 February 1994, and not    its subsequent conference version \cite{V_1994}}
 Alternative proximity measures 
 (such as the relative Renyi entropy and Hellinger distance) instead of the Kullback-Leibler informational divergence in application to the anisotropy functional are discussed in \cite{C_2018}.\footnote{where there is also an incorrect use of concepts of the  anisotropy-based theory: for example, in \cite[Eq. (1)]{C_2018} and throughout the paper, the anisotropy of a random vector, which is scale invariant and hence not positively homogeneous (in contrast to any norm), is confused with the anisotropic norm of linear operators with respect to random inputs}

The discrete-time setting is essential for the anisotropy-based theory, which
 stems from the anisotropy functionals using 
  finite-dimensional spheres and finite segments of random sequences. At the same time, entropy-theoretic constructs are also known for infinite-dimensional objects such as
 diffusion processes. For example, the relative entropy (with respect to the Wiener measure) is correctly defined for the probability distribution of a diffusion process governed by a stochastic differential equation (SDE) with a nonzero drift (satisfying mean square integrability conditions) and the identity diffusion matrix. The preservation of the diffusion matrix of the standard Wiener process is essential for the absolute continuity of measures
 and applicability of the Girsanov theorem \cite{G_1960} and thus plays a part in 
  relative entropy  formulations of statistical uncertainty for continuous time stochastic control  systems driven by diffusion processes \cite{CR_2007,PUJ_2000}.

The present paper discusses one of possible ways of extending the anisotropy-based approach (which underlies stochastic minimax formulations of robust filtering and control using worst-case RMS gains under entropy-theoretic constraints in terms of anisotropy) to linear continuous time invariant (LCTI) systems driven by statistically uncertain Gaussian diffusion processes. To this end, we employ a parametric family of conformal maps (specified by an auxiliary time scale parameter and related to the Cayley transform  \cite{S_1992}) between the unit disk and the right half-plane for the discrete and continuous time transfer functions. Up to a multiplicative constant, these maps are identical to Tustin's transform \cite{B_2009} for converting LCTI systems to linear discrete-time invariant (LDTI) systems.  This results in a subsidiary LDTI system which lends itself to application of the anisotropy-based theory. The spectral densities of the input and output of this subsidiary system are related to those of filtered versions of the  original continuous-time processes (with the time scale parameter specifying the transient time of the filter) due to the relation between the spectral density of the Ornstein-Uhlenbeck (OU) process \cite{KS_1991} and the logarithmic derivative of the Cayley map.

The conformal correspondence is closely related to the presence of a denominator (which improves  integrability in the continuous time case) in the inner-outer factorizability condition for spectral densities (describing the physical realizability of such a density with the aid of a stable causal shaping filter) \cite{GS_2004,W_1972}. This is one of the reasons why discrete-time results (many of which rely on boundedness of the frequency range) cannot merely be adopted without due modification (in particular, without taking the denominator into account)   for the continuous-time case, where the processes may contain arbitrarily fast components  accommodated by the infinite frequency range.  Ignoring this qualitative distinction leads to divergent integrals in \cite{B_2018,BK_2018},\footnote{which makes the definition of ``$\sigma$-entropy'' as the quantity ``$-\frac{1}{2} \int_\mR \ln\det \frac{S(\omega)}{\int_\mR \Tr S(\lambda) \rd \lambda}\rd \omega$'' in \cite[Eq. (5)]{B_2018} and \cite[Eq. (19)]{BK_2018} incorrect because of nonexistence of an integrable spectral density $S: \mR\to \mC^{m\x m}$ on the real axis with an integrable log-determinant in view of the relations $\int_{-b}^b(\Tr S(\omega)-\ln\det S(\omega))\rd \omega \> 2bm \to +\infty$ as $b\to +\infty$, whereby $\int_\mR \Tr S(\omega)\rd \omega <+\infty$ implies $\int_\mR \ln\det S(\omega)\rd \omega  =-\infty$, and the same integral divergence issue makes \cite[Theorem~1]{B_2018} and  \cite[Theorems~3, 4]{BK_2018} incorrect in addition to the log-determinant being undefined for rank-one matrices of orders greater than one in \cite{B_2018}} and the correction of some of the errors mentioned above is one of the secondary purposes of the present paper.

The paper is organised as follows.
Section~\ref{sec:sys} specifies a class of LCTI systems governed by linear SDEs driven by Ito processes.  
Section~\ref{sec:filt} describes filtered versions of the underlying processes for such a system.
Section~\ref{sec:RMS} defines the RMS gain of the system in terms of the filtered input and output.
Section~\ref{sec:IWN} calculates the RMS gain in the nominal case of isotropic white-noise disturbances.
Section~\ref{sec:conf} specifies a parameter-dependent conformal correspondence to a discrete-time system with the same RMS gain.
Section~\ref{sec:an} defines a two-parameter norm of the underlying system in terms of the anisotropic norm of the effective LDTI system.
Section~\ref{sec:state} describes state-space equations for the continuous-time extension of the anisotropic norm along with the structure of the worst-case disturbance.
Section~\ref{sec:ex} provides a numerical example which demonstrates the computation of the norm.
Section~\ref{sec:conc} makes concluding remarks.

\section{Linear stochastic systems being considered}
\label{sec:sys}

Consider an LCTI 
system with an $\mR^m$-valued input $W$ and $\mR^p$-valued output $Z$. It is assumed that $W$ and $Z$ are Ito processes \cite{KS_1991} with respect to a common filtration 
on the underlying probability space and are related in a causal linear fashion:
\begin{equation}
\label{ZfW}
    Z(t)
    =
    \int_{-\infty}^t
    f(t-\tau)
    \rd W(\tau)
    +
        f_0 W(t),
    \qquad
    t\in \mR.
\end{equation}
Here, $f_0 \in \mR^{p\x m}$ is a static gain matrix, and $f: \mR_+ \to \mR^{p\x m}$ is the decaying part of the  step response of the system, with $\mR_+:= [0,+\infty)$ the set of nonnegative real numbers. In order to ensure the existence of the Ito integral in (\ref{ZfW}) in the case when $W$ is a standard Wiener process \cite{KS_1991} in $\mR^m$, the function $f$ is assumed to be square integrable. In this case,  the Ito isometry yields
\begin{equation}
\label{Itiso}
    \bE
    \Big(
    \Big|
    \int_{-\infty}^t
    f(t-\tau)
    \rd W(\tau)
    \Big|^2
    \Big)
    =
    \int_{\mR_+}
    \|f(\tau)\|_\rF^2 \rd \tau
    <
    +\infty
\end{equation}
for any $t \in \mR$,
where $\bE(\cdot)$ is expectation, and $\|\cdot\|_{\rF}$ is the Frobenius norm of matrices \cite{HJ_2007}. If
$f(\tau)$ is a linear combination of quasi-polynomials  \cite{F_1985}  which decay exponentially fast as $\tau\to +\infty$, there exist constant matrices $A \in \mR^{n\x n}$, $B\in \mR^{n\x m}$, $C \in \mR^{p\x n}$, with $A$ Hurwitz, such that
\begin{equation}
\label{fABC}
    f(\tau)
    =
    CA^{-1}\re^{\tau A}B,
    \qquad
    \tau \>0.
\end{equation}
Then the stochastic differential of the output process $Z$  takes the form
\begin{align}
\nonumber
    \rd Z(t)
    & =
    \Big(
    \int_{-\infty}^t
    f'(t-\tau)
    \rd W(\tau)
    \Big)
    \rd t
    +
    (f(0)+f_0) \rd W(t)\\
\label{dZ}
    & =
    CX(t)\rd t
    +
    D \rd W(t),
\end{align}
where
\begin{equation}
\label{XW}
    X(t)
    :=
    \int_{-\infty}^t
    \re^{(t-\tau)A}
    B
    \rd W(\tau),
\end{equation}
and
\begin{equation}
\label{DABC}
    D := f_0 + C A^{-1} B
\end{equation}
is the feedthrough $(p\x m)$-matrix.
The relation (\ref{dZ}) can be obtained by using the integration-by-parts  formula
$
    \int_{-\infty}^t
    f(t-\tau)
    \rd W(\tau)
    =
    f(0)W(t)
    +
    \int_{-\infty}^t
    f'(t-\tau)
    W(\tau)
    \rd \tau
$,
which, in view of (\ref{ZfW}), shows that the time derivative of (\ref{fABC}), given by
$
    f'(\tau) = C\re^{\tau A}B$
for all $
    \tau\>0
$,
is the impulse response for the map $W\mapsto Z$. The $\mR^n$-valued Ito process $X$ in (\ref{XW}) is an internal state of the system and satisfies the linear SDE
\begin{equation}
\label{dX}
    \rd X
    =
    AX \rd t  + B \rd W
\end{equation}
(the time arguments are omitted for brevity), whose solution is given by 
$    X(t)
    =
    \re^{(t-s)A}X(s)
    +
    \int_s^t \re^{(t-\tau) A}B\rd W(\tau)
$ for all $
    t\>s
$. 
The system state can be essentially  infinite-dimensional, as in the case of systems with delay \cite{KN_1986}. For example, such a system arises if the drift $AX(t)$ in (\ref{dX}) is replaced with a linear function of the past history of the process $X$:
\begin{equation}
\label{dXdelay}
    \rd X(t)
    =
    \Big(
    \int_{\mR_+}
    \vartheta(\rd \tau) X(t-\tau)
    \Big)
    \rd t
    +
    B
    \rd W(t),
\end{equation}
where $\vartheta$ is an $\mR^{n\x n}$-valued countably additive measure on the $\sigma$-algebra of Borel subsets of $\mR_+$, and the integral is understood pathwise. The SDE (\ref{dX}) is a particular case of (\ref{dXdelay}) with an atomic measure $\vartheta$, concentrated at the origin as  $\vartheta(S) = \chi_S(0) A$, with $\chi_S$ the indicator function of a set $S$. More generally, if $\vartheta$ is of bounded support, with $\supp \vartheta \subset [0,r]$ for some $r\>0$, the effective state of the system at time $t$ is the past history of $X$ over the time interval $[t-r,t]$.

Returning to the SDEs (\ref{dZ}), (\ref{dX}), note that there also are alternative ways to model the input-output operator. For example, the output channel can be in the form $Z = CX$, which is different from (\ref{dZ}). However, a convenient feature of the representation being considered is that both the input $W$ and the output $Z$ enter the SDEs (\ref{dZ}), (\ref{dX}) in a unified fashion (through their Ito increments), thus allowing such systems to be easily concatenated.

We will be concerned with a setting where the $\mR^m$-valued  input $W$ does not necessarily obey the  nominal model assumption of being a standard Wiener process (and hence, (\ref{Itiso}) is no longer valid). Rather, $W$ is assumed to be an Ito process with respect to an underlying standard Wiener process $V$ in $\mR^m$ (and hence, so also are
the $\mR^n$-valued state $X$ and the $\mR^p$-valued output $Z$ driven by $W$).
For example, the process $W$ can be modelled as the output of a causal LCTI system (playing the role of a shaping filter):
\begin{equation}
\label{WV}
    W(t)
    =
    \int_{-\infty}^t
    g(t-\tau) \rd V(\tau)
    +
    g_0 V(t)
\end{equation}
with a static gain matrix $g_0 \in \mR^{m \x m}$ and a square integrable decaying part $g: \mR_+ \to \mR^{m\x m}$ of the step response; see Fig.~\ref{fig:FG}.
\begin{figure}[htbp]
\unitlength=0.6mm
\linethickness{0.4pt}
\centering
\begin{picture}(110,15.00)
    \put(35,0){\framebox(10,10)[cc]{{$F$}}}
    \put(65,0){\framebox(10,10)[cc]{{$G$}}}
    \put(35,5){\vector(-1,0){20}}
    \put(65,5){\vector(-1,0){20}}
    \put(95,5){\vector(-1,0){20}}
    \put(10,5){\makebox(0,0)[cc]{$Z$}}
    \put(100,5){\makebox(0,0)[cc]{$V$}}

    \put(55,10){\makebox(0,0)[cc]{$W$}}

\end{picture}
\caption{The LCTI system $F$ with the output $Z$ and input $W$ produced by a causal shaping filter $G$ from the standard Wiener process $V$.}
\label{fig:FG}
\end{figure}
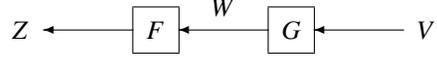
 Similarly to the system itself, the shaping filter  can also have a  finite-dimensional state-space representation
 \begin{equation}
 \label{dxi}
   \rd \xi = \alpha \xi \rd t + \beta \rd V,
   \qquad
   \rd W = c \xi \rd t + d \rd V,
 \end{equation}
 where $\xi$ is an $\mR^\nu$-valued Ito process, and $\alpha\in \mR^{\nu\x \nu}$, $\beta \in \mR^{\nu\x m}$, $c \in \mR^{m \x \nu}$, $d \in \mR^{m \x m}$, with $\alpha$ Hurwitz. In this case, similarly to (\ref{fABC}), (\ref{DABC}), the static gain matrix and the decaying part of the step response in (\ref{WV}) are given by
 \begin{equation}
 \label{gg}
    g_0 = d - c \alpha^{-1} \re^{\tau \alpha} \beta,
    \quad
    g(\tau) = c \alpha^{-1} \re^{\tau \alpha} \beta,
    \quad
    \tau \> 0.
 \end{equation}
Being a Gaussian Markov process with continuous sample paths and independent stationary increments satisfying
$
    \bE((V(t)-V(\tau))(V(t)-V(\tau))^\rT) = |t-\tau| I_m
$
(with $I_m$ the identity matrix of order $m$)
for any moments of time $t$, $\tau$, the standard Wiener process $V$ 
has almost surely nowhere differentiable sample paths of unbounded variation (and finite quadratic variation) over any (bounded) time interval \cite{KS_1991}. Such
nonsmoothness is inherited by the Ito processes $Z$, $X$, $W$ in (\ref{dZ}), (\ref{XW}), (\ref{WV}). Furthermore, in the case of nonzero static gain matrices $f_0$, $g_0$, the variances of $Z$, $W$ grow unboundedly over time (despite the stability of the system due to the matrix $A$ being Hurwitz).

\section{Filtered signals and transfer functions}
\label{sec:filt}

As a more realistic (in the sense of the variances) model of the input and output processes of the system and the shaping filter  in Fig.~\ref{fig:FG}, consider their filtered versions governed by the identical SDEs
\begin{align}
\label{dVT}
    \rd V_T
    & = - \Omega V_T\rd t + \sqrt{2\Omega }\rd V,\\
\label{dWT}
    \rd W_T
    & = - \Omega W_T\rd t + \sqrt{2\Omega }\rd W,\\
\label{dZT}
    \rd Z_T
    & = - \Omega Z_T\rd t + \sqrt{2\Omega}\rd Z,
\end{align}
where $T>0$ is an auxiliary time scale parameter which specifies the frequency
\begin{equation}
\label{Omega}
    \Omega
    :=
    \frac{1}{T}.
\end{equation}
Since $V$ is a standard Wiener process (to which $W$ in (\ref{WV}) reduces in the case of $g=0$, $g_0 = I_m$), the SDE (\ref{dVT})\footnote{when it is initialised at zero state in the infinitely distant past or at the invariant Gaussian distribution at time $t=0$} generates an OU process \cite{KS_1991} in $\mR^m$, which is a stationary Gaussian 
diffusion process with zero mean and covariance function
$
    \bE(V_T(t)V_T(\tau)^\rT)
    =
    \re^{-\Omega |t-\tau|}
    I_m
$ for all $
    t,\tau \in \mR
$,
so that the variance $\bE(|V_T(t)|^2) = m$ remains constant in time, and $T$ quantifies the correlation time of the process.\footnote{a characteristic time of decay in the correlation between the values of the process at different moments of time} Since we will be concerned with the infinite-horizon asymptotic behaviour of stable systems, it will be convenient to endow (\ref{dVT})--(\ref{dZT}) with zero initial conditions: $V_T(0) = 0$, $W_T(0) = 0$, $Z_T(0) = 0$.

The asymptotic input-output properties of the system, described by  (\ref{dZ}), (\ref{dX})  with a Hurwitz matrix $A$, and the shaping filter (\ref{WV}), can be formulated in terms of the transfer functions
\begin{equation}
\label{Ftrans}
    F(s):= C(sI_n - A)^{-1} B + D
\end{equation}
(analytic in a neighbourhood of the closed right half-plane $\mC_+:= \{s \in \mC:\ \Re s \>0\}$), and
\begin{equation}
\label{Gtrans}
    G(s)
    :=
    \int_0^{+\infty}
    \re^{-st}
    g'(t)
    \rd t
    +
    d =
    s
    \int_0^{+\infty}
    \re^{-st}
    g(t)
    \rd t
    +
    g_0.
\end{equation}
Here, for simplicity, the shaping filter is also assumed to be a stable system with a finite-dimensional state $\xi$ in (\ref{dxi}), (\ref{gg}),
so that 
(\ref{Gtrans}) is a rational function
\begin{equation}
\label{Gabcd}
  G(s) = c(sI_\nu - \alpha)^{-1} \beta + d,
\end{equation}
analytic in a neighbourhood of $\mC_+$.  Similarly to deterministic linear systems, the transfer functions  $F$, $G$  specify the linear relations
\begin{equation}
\label{ZFW_WGV}
    \wh{Z}(s) = F(s) \wh{W}(s),
    \qquad
    \wh{W}(s) = G(s) \wh{V}(s)
\end{equation}
between the Laplace transforms of the input and output in the case $X(0)=0$, $\xi(0) = 0$:
\begin{align}
\label{Vhat}
    \wh{V}(s)
    & :=
    \int_0^{+\infty}
    \re^{-st}
    \rd V(t),\\
\label{What}
    \wh{W}(s)
    & :=
    \int_0^{+\infty}
    \re^{-st}
    \rd W(t),\\
\label{Zhat}
    \wh{Z}(s)
    & :=
    \int_0^{+\infty}
    \re^{-st}
    \rd Z(t),
\end{align}
where the Ito integrals are well-defined for $\Re s >0$. The same relations
\begin{equation}
\label{ZTFWT_WTGVT}
    \wh{Z}_T(s) = F(s) \wh{W}_T(s),
    \qquad
    \wh{W}_T(s) = G(s) \wh{V}_T(s)
\end{equation}
hold for the usual Laplace transforms
\begin{align}
\label{VThat}
    \wh{V}_T(s)
    & :=
    \int_0^{+\infty}
    \re^{-st}
    V_T(t)
    \rd t
    =
    \Phi(s)\wh{V}(s),\\
\label{WThat}
    \wh{W}_T(s)
    & :=
    \int_0^{+\infty}
    \re^{-st}
    W_T(t)
    \rd t
    =
    \Phi(s)\wh{W}(s),\\
\label{ZThat}
    \wh{Z}_T(s)
    & :=
    \int_0^{+\infty}
    \re^{-st}
    Z_T(t)
    \rd t
    =
    \Phi(s)\wh{Z}(s)
\end{align}
of the filtered processes $V_T$, $W_T$, $Z_T$ in (\ref{dVT})--(\ref{dZT}); see Fig.~\ref{fig:FGT}.
Here,
\begin{equation}
\label{Phi}
    \Phi(s)
    :=
    \frac{\sqrt{2\Omega}}{s + \Omega}
    =
    \frac{\sqrt{2T}}{1 + Ts}
\end{equation}
specifies the common transfer function 
for these low-pass filters with the cutoff frequency $\Omega$  in (\ref{Omega}). Indeed, (\ref{ZTFWT_WTGVT}) follows from  (\ref{ZFW_WGV}), (\ref{VThat})--(\ref{ZThat}) since the scalar-valued function $\Phi$ commutes with any transfer matrix.  Although
$V_T$, $W_T$, $Z_T$ inherit nonsmoothness of sample paths from $V$, $W$, $Z$, they capture (approximately) only relatively slow components of the original processes (with frequencies $\ll \Omega$, or equivalently,  time scales $\gg T$).
\begin{figure}[htbp]
\unitlength=0.6mm
\linethickness{0.4pt}
\centering
\begin{picture}(110,45.00)
    \put(35,0){\framebox(10,10)[cc]{{$F$}}}
    \put(50,15){\framebox(10,10)[cc]{{$\Phi$}}}
    \put(80,15){\framebox(10,10)[cc]{{$\Phi$}}}
    \put(25,35){\vector(0,-1){10}}
    \put(55,35){\vector(0,-1){10}}
    \put(85,35){\vector(0,-1){10}}
    \put(55,15){\line(0,-1){10}}
    \put(85,15){\line(0,-1){10}}
    \put(25,15){\line(0,-1){10}}
    \put(65,5){\vector(-1,0){20}}
    \put(95,5){\vector(-1,0){20}}
    \put(35,5){\vector(-1,0){20}}
    \put(20,15){\framebox(10,10)[cc]{{$\Phi$}}}
    \put(35,30){\framebox(10,10)[cc]{{$F$}}}
    \put(65,30){\framebox(10,10)[cc]{{$G$}}}
    \put(65,0){\framebox(10,10)[cc]{{$G$}}}
    \put(35,35){\vector(-1,0){20}}
    \put(65,35){\vector(-1,0){20}}
    \put(95,35){\vector(-1,0){20}}

    \put(10,35){\makebox(0,0)[cc]{$Z$}}
    \put(10,5){\makebox(0,0)[cc]{$Z_T$}}
    \put(100,35){\makebox(0,0)[cc]{$V$}}
    \put(100,5){\makebox(0,0)[cc]{$V_T$}}

    \put(55,40){\makebox(0,0)[cc]{$W$}}
    \put(55,0){\makebox(0,0)[cc]{$W_T$}}

\end{picture}
\caption{The LCTI system $F$ with the output $Z$ and input $W$ produced by a causal shaping filter $G$ from the standard Wiener process $V$, and their filtered versions $Z_T$, $W_T$, $V_T$ obtained using the low-pass filter $\Phi$. The processes  $Z_T$, $W_T$, $V_T$ are related by the same input-output operators $F$, $G$ as $Z$, $W$, $V$ (cf. (\ref{ZFW_WGV})--(\ref{ZThat})).}
\label{fig:FGT}
\end{figure}
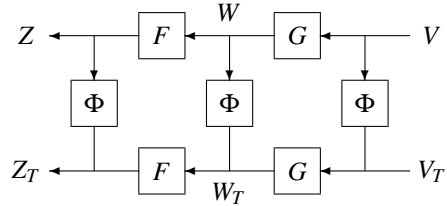

\section{Root-mean-square gain for filtered processes}
\label{sec:RMS}

For a given time scale parameter $T>0$, the RMS gain of the system under consideration
can be quantified
in terms of the filtered input and output $W_T$, $Z_T$  in (\ref{dWT}), (\ref{dZT}) by
\begin{equation}
\label{Fgain}
    \sn F\sn
    :=
    \limsup_{t\to +\infty}
    \sqrt{
    \frac{\int_0^t \bE (|Z_T(\tau)|^2)\rd \tau }
    {\int_0^t \bE (|W_T(\tau)|^2)\rd \tau}
    }
\end{equation}
(the trivial case of zero inputs is not considered). 
The following theorem shows that the upper limit in (\ref{Fgain}) exists as a limit. Its formulation employs an auxiliary function
\begin{align}
\nonumber
  \Sigma(\omega)
  & :=
  |\Phi(i\omega)|^2\\
\label{Sigma}
  & =
  \frac{2\Omega }{\Omega^2 + \omega^2}
  =
  \frac{2T }{1 +(\omega T)^2},
  \qquad
  \omega \in \mR,
\end{align}
which is related to (\ref{Phi})
and specifies the spectral density $\Sigma(\omega) I_m$ for the OU process $V_T$ in $\mR^m$. 
\begin{theorem}
\label{th:RMS}
Suppose the system (\ref{dZ}), (\ref{dX}) and the shaping filter (\ref{dxi}) are stable. Then the RMS gain (\ref{Fgain}) holds as a limit and is computed as
\begin{align}
\nonumber
    \sn F\sn
    &=
    \sqrt{
    \frac{\lim_{t\to +\infty} \bE(|Z_T(t)|^2)}{\lim_{t\to +\infty} \bE(|W_T(t)|^2)}}\\
\label{limFgain}
    & =
    \sqrt{
    \frac{\int_{\mR} \Sigma(\omega) \bra \Lambda(\omega),  S(\omega)\ket \rd \omega}
    {\int_{\mR} \Sigma(\omega) \Tr S(\omega)\rd \omega}}.
\end{align}
Here, the function $\Sigma$ is given by (\ref{Sigma}),
$\bra \cdot, \cdot\ket $ is the Frobenius inner product of matrices, and the functions $\Lambda, S:\mR \to \mH_m^+$ are associated with the transfer functions (\ref{Ftrans})--(\ref{Gabcd}) as
\begin{align}
\label{FF}
    \Lambda(\omega)
    & :=
    F(i\omega)^*F(i\omega),\\
\label{GG}
    S(\omega)
    & :=
    G(i\omega)G(i\omega)^*
\end{align}
and take values in the set $\mH_m^+$ of complex positive semi-definite Hermitian matrices of order $m$, with $(\cdot)^* := (\overline{(\cdot)})^\rT$ the complex conjugate transpose.
\end{theorem}
\begin{proof}
In view of the initial condition $W_T(0) = 0$,
the solution of the SDE (\ref{dWT}) takes the form
\begin{align}
\nonumber
    W_T(t)
    & =
    \int_0^t
    \phi(t-\tau)
    \rd W(\tau)\\
\nonumber
    & =
    \int_0^t
    \phi(t-\tau)
    (c\xi(\tau)\rd \tau + d \rd V(\tau))\\
\nonumber
    & =
    \int_0^t
    \phi(t-\tau)
    \Big(
        c\int_0^\tau \re^{(\tau-u) \alpha} \beta \rd V(\rd u)\rd \tau + d \rd V(\tau)
    \Big)\\
\label{WTV}
    & =
    \int_0^t
    h(t-u)
    \rd V(u),
\end{align}
where
\begin{align}
\nonumber
    h(\tau)
    & :=
    (\phi* \psi)(\tau)\\
\nonumber
    & =
    \int_0^\tau
    \phi(\tau-u)
    \psi(u)
    \rd u\\
\nonumber
    & =
    \int_0^\tau
    \phi(\tau-u)
    c
    \re^{u \alpha} \beta
    \rd u
    +
    \phi(\tau)
    d,
    \qquad
    \tau\> 0,
\end{align}
is the convolution of the impulse responses of the low-pass filter  (\ref{dWT}) and the shaping filter (\ref{dxi}):
$
    \phi(\tau):=
    \sqrt{2\Omega}
    \re^{-\Omega \tau}
$, $
    \psi(\tau)
    :=
    c \re^{\tau \alpha}
    \beta + d \delta(\tau)
$,
$
    \tau\> 0
$
(with $\delta$ the Dirac delta function \cite{V_2002}).
Since $V$ is a standard Wiener process in $\mR^m$, application of the Ito isometry to (\ref{WTV}) leads to
\begin{equation}
\label{varWT}
    \bE(|W_T(t)|^2)
    =
    \int_0^t
    \|h(u)\|_\rF^2
    \rd u,
    \qquad
    t\>0.
\end{equation}
Due to the square integrability of the transfer function $\Phi G$  (which is the Laplace transform of $h$) over the imaginary axis and the Plancherel theorem, the function $h$ is square integrable over $\mR_+$, and (\ref{varWT}) leads to the monotonic convergence
\begin{align*}
\nonumber
    \lim_{t\to +\infty}&
    \bE(|W_T(t)|^2)
     =
    \int_{\mR_+}
    \|h(u)\|_\rF^2
    \rd u\\
\nonumber
    & =
    \frac{1}{2\pi}
    \int_{\mR}
    |\Phi(i\omega)|^2
    \|G(i\omega)\|_\rF^2
    \rd \omega\\
    &=
    \frac{1}{2\pi}
    \int_{\mR}
    \Sigma(\omega)
    \Tr S(\omega)
    \rd \omega
\end{align*}
in view of (\ref{Sigma}), (\ref{GG}). The same limit is shared by the Cesaro means
\begin{align}
\nonumber
    \lim_{t\to +\infty}
    \Big(
        \frac{1}{t}&
        \int_0^t
        \bE(|W_T(\tau)|^2)
        \rd \tau
    \Big)
    =
    \lim_{t\to +\infty}
    \bE(|W_T(t)|^2)\\
\label{limvarWT1}
    & =
    \frac{1}{2\pi}
    \int_{\mR}
    \Sigma(\omega)
    \Tr S(\omega)
    \rd \omega.
\end{align}
In application to the filtered output $Z_T$,  a similar reasoning shows that
\begin{align}
\nonumber
    \lim_{t\to +\infty}
    \Big(
        \frac{1}{t}&
        \int_0^t
        \bE(|Z_T(\tau)|^2)
        \rd \tau
    \Big)
    =
    \lim_{t\to +\infty}
    \bE(|Z_T(t)|^2)
\\
\nonumber
    & =
    \frac{1}{2\pi}
    \int_{\mR}
    \Sigma(\omega)
    \|F(i\omega)G(i\omega)\|_\rF^2
    \rd \omega\\
\nonumber
    & =
    \frac{1}{2\pi}
    \int_{\mR}
    \Sigma(\omega)
    \Tr(F(i\omega)S(\omega) F(i\omega)^*)
    \rd \omega\\
\label{limvarZT1}
    & =
    \frac{1}{2\pi}
    \int_{\mR}
    \Sigma(\omega)
    \bra
        \Lambda(\omega),
        S(\omega)
    \ket
    \rd \omega,
\end{align}
where (\ref{FF}) is also used.  A combination of (\ref{limvarWT1}), (\ref{limvarZT1}) implies the existence of a limit in (\ref{Fgain}) which is given by (\ref{limFgain}).
\end{proof}

The RMS gain $\sn F\sn$ in  (\ref{limFgain}) is a semi-norm of the system transfer function $F$ in (\ref{Ftrans}). It is well-defined for any bounded (and not necessarily integrable) spectral density $S$ (the trivial case when $S=0$ almost everywhere is excluded from consideration). This is secured by the presence of the  integrable function $\Sigma$ from (\ref{Sigma}) as a factor in the integrands. The dependence of the RMS gain on $S$ will be indicated by the subscript as $\sn F \sn = \sn F\sn_S$. 
Note that $\sn F\sn_S$ is invariant under the scaling $S\mapsto \kappa S$ with an arbitrary constant $\kappa>0$ (which cancels out in the numerator and denominator in (\ref{limFgain})):
\begin{equation}
\label{kappaS}
    \sn F\sn_{\kappa S}
    =
     \sn F\sn_S.
\end{equation}
Also, if the system $F$ is isometric (that is, inner) up to a multiplicative constant, so that the function (\ref{FF}) satisfies $\Lambda(\omega) = \lambda I_m$ for all $\omega \in \mR$ for some constant $\lambda>0$, then the RMS gain reduces to
\begin{equation}
\label{Fiso}
    \sn F \sn_S
    =
    \sqrt{
    \frac{\int_{\mR} \Sigma(\omega) \bra \lambda I_m, S(\omega)\ket \rd \omega}
    {\int_{\mR} \Sigma(\omega) \Tr S(\omega)\rd \omega}}
    =
    \sqrt{\lambda}
\end{equation}
and does not depend on the spectral density $S$. In what follows, such systems $F$ will be referred to as \emph{round} systems (with all the other systems being called \emph{nonround}).
In general,  application of the inequalities $\lambda_{\min}(L)\Tr M \< \bra L,M\ket\< \lambda_{\max}(L)\Tr M$ (for Hermitian matrices $L$, $M$, with $M\succcurlyeq 0$) leads to
\begin{align}
\nonumber
    \inf_{\omega \in \mR}
    \sqrt{\lambda_{\min}(\Lambda(\omega))}
    & \<
    \sn F \sn_S\\
\label{RMSbounds}
    &\<
    \sup_{\omega \in \mR}
    \sqrt{\lambda_{\max}(\Lambda(\omega))}
    =
    \|F\|_{\infty},
\end{align}
where $\lambda_{\min}(\cdot)$, $\lambda_{\max}(\cdot)$ are the smallest and largest eigenvalues of a Hermitian matrix,
and the right-hand side is the $\cH_\infty$-norm \cite{F_1987} of the transfer function, so that (\ref{Fiso}) is a particular case of (\ref{RMSbounds}). If the input  and output dimensions of the system satisfy $m>p$,  then the matrix $\Lambda(\omega)$ in (\ref{FF}) is singular (at every frequency $\omega \in \mR$) and the left-hand side of (\ref{RMSbounds}) vanishes. For arbitrary dimensions, the RMS gain $\sn F \sn_S$ can approach any intermediate value in (\ref{RMSbounds}) by an appropriate choice of the spectral density $S$.

\begin{theorem}
\label{th:inter}
Suppose the system (\ref{dZ}), (\ref{dX}) is stable. Then for any $\mu$ satisfying
$    \inf_{\omega \in \mR}
    \sqrt{\lambda_{\min}(\Lambda(\omega))} \< \mu \< \|F\|_{\infty}$ in (\ref{RMSbounds}) and any $\eps>0$,  there exists a rational spectral density $S$ (associated with a stable shaping filter (\ref{dxi})) such that $|\sn F \sn_S-\mu|< \eps$.
\end{theorem}
\begin{proof}
With any $\gamma >0$, a frequency $\omega_0 \in \mR$ and a unit vector $u \in \mC^m$, we associate a rational spectral  density $S_{\gamma}: \mR\to \mH_m^+$:
\begin{equation}
\label{Seps}
    S_{\gamma}(\omega)
    :=
    \frac{\gamma}{\pi}
    \Big(
    \frac{1}{\gamma^2 + (\omega-\omega_0)^2}uu^*
    +
    \frac{1}{\gamma^2 + (\omega+\omega_0)^2}\overline{u}u^\rT
    \Big),
\end{equation}
which satisfies $\int_{\mR}\Tr S(\omega) \rd \omega = 2$. Due to the weak convergence \cite{B_1968} of the Cauchy distribution with the probability density function $    \frac{1}{\pi}
    \frac{\gamma}{\gamma^2 + (\omega-\omega_0)^2}$ to the atomic probability measure,   concentrated at $\omega_0$,  as $\gamma\to 0+$ (and similarly for the opposite frequency $-\omega_0$), the spectral density (\ref{Seps}) is  convergent in the distributional sense \cite{V_2002} as
\begin{equation}
\label{limSeps}
    \lim_{\gamma\to 0+}
    S_{\gamma}(\omega)
    =
    \delta(\omega-\omega_0) uu^*  + \delta(\omega+\omega_0)\overline{u}u^\rT.
\end{equation}
 Since the functions $\Sigma$,  $\Lambda$ in (\ref{Sigma}), (\ref{FF}) are bounded and continuous, (\ref{limSeps}) implies that
\begin{align}
\label{lim1}
    \lim_{\gamma\to 0+}
    \int_\mR
    \Sigma(\omega)
    \bra
        \Lambda(\omega),
        S_{\gamma}(\omega)
    \ket
    \rd \omega
    & =
    2
    \Sigma(\omega_0)
    \|u\|_{\Lambda(\omega_0)}^2,\\
\label{lim2}
    \lim_{\gamma\to 0+}
    \int_\mR
    \Sigma(\omega)
    \Tr
        S_{\gamma}(\omega)
    \rd \omega
    & =
    2
    \Sigma(\omega_0),
\end{align}
where $\|v\|_M:= |\sqrt{M} v| = \sqrt{v^* M v}$ is a weighted Euclidean semi-norm of a complex vector $v$ specified by a positive semi-definite Hermitian matrix $M$. Here, use is also made of the assumption $|u|=1$ and the property
\begin{equation}
\label{Lamom}
    \Lambda(-\omega)
    =
    \overline{\Lambda(\omega)}
    =
    \Lambda(\omega)^\rT,
\end{equation}
whereby $\Lambda(-\omega)$ is isospectral to $\Lambda(\omega)$, and   $\|\overline{u}\|_{\Lambda(-\omega)} = \|u\|_{\Lambda(\omega)}$.  In view of (\ref{lim1}), (\ref{lim2}), the corresponding RMS gain in (\ref{limFgain}) satisfies
\begin{equation}
\label{Del1}
    \lim_{\gamma\to 0+}
    \sn F\sn_{S_{\gamma}}
    =
    \|u\|_{\Lambda(\omega_0)}
    \in
    \Delta(\omega_0),
\end{equation}
where $\Delta$ is a set-valued map which maps a frequency $\omega \in \mR$ to the interval
\begin{equation}
\label{Del2}
    \Delta(\omega)
    :=
    \big[
        \sqrt{\lambda_{\min}(\Lambda(\omega))}, \,
        \sqrt{\lambda_{\max}(\Lambda(\omega))}
    \big].
\end{equation}
Any given point of $\Delta(\omega_0)$ in (\ref{Del1}) is achievable by an appropriate choice of a unit vector $u\in \mC^m$ since $\{\|u\|_{\Lambda(\omega)}: u \in \mC^m, |u|=1 \}=\Delta(\omega)$ for any $\omega \in \mR$.
With the endpoints of this interval being continuous even functions of $\omega$ (the continuity is inherited from $\Lambda$), the map $\Delta$ in (\ref{Del2}) covers the interior of the interval in (\ref{RMSbounds}) in the sense that
$
    \big(
    \inf_{\omega \in \mR}
    \sqrt{\lambda_{\min}(\Lambda(\omega))}, \,
    \|F\|_{\infty}
    \big)
    \subset
    \bigcup_{\omega\in \mR}
    \Delta(\omega)
$.
Therefore, for any $\mu$ from the interval in (\ref{RMSbounds}) and any $\eps>0$, there exists $\omega_0 \in \mR$ and a unit vector $u \in \mC^m$ such that $|\|u\|_{\Lambda(\omega_0)} - \mu|< \eps$. In view of (\ref{Del1}), the corresponding spectral density (\ref{Seps}) delivers the RMS gain  which satisfies $|\sn F\sn_{S_\gamma} -  \|u\|_{\Lambda(\omega_0)}| < \eps$ for all sufficiently small $\gamma >0$, and hence, $|\sn F\sn_{S_\gamma} -  \mu| < 2\eps$
by the triangle inequality.
\end{proof}

It follows from Theorem~\ref{th:inter} that (in the absence of specific additional constraints on the spectral density $S$),  the second inequality in (\ref{RMSbounds}) cannot be improved since 
    $\sup_S
    \sn F\sn_S
    =
    \|F\|_{\infty}$, 
and moreover, the supremum can be restricted (without affecting its value) to the class of rational spectral densities $S: \mR \to \mH_m^+$. The proof of the theorem provides a particular way to construct a maximizing sequence of spectral densities $S_\gamma$ in  (\ref{Seps}), with $u$ being a  unit eigenvector of the matrix $\Lambda(\omega_0)$ associated with its largest eigenvalue. Such  spectral densities exhibit ``energy concentration'' about certain frequencies and in certain directions in $\mC^m$.

\section{RMS gain with respect to isotropic white-noise inputs}
\label{sec:IWN}

As opposed to the energy concentration used in the proof of Theorem~\ref{th:inter},
consider an isotropic white-noise case when the spectral density of the input $W$ is a constant scalar matrix:
\begin{equation}
\label{S0}
    S(\omega) = \kappa I_m,
    \qquad
    \omega \in \mR,
\end{equation}
where $\kappa>0$ is a scalar parameter. In this case, $W$ is a standard Wiener process up to the multiplicative constant $\sqrt{\kappa}$ (and can be obtained from $V$ by letting $g=0$, $g_0= \sqrt{\kappa} I_m$ in (\ref{WV})), so that $W_T$ is an OU process.  Then the RMS gain takes the form
\begin{align}
\nonumber
    \sn F\sn_{\kappa I_m}
    & =
    \sn F\sn_{I_m}
    = \sqrt{
    \frac{\int_{\mR} \Sigma(\omega) \Tr \Lambda(\omega)\rd \omega}
    {m \int_{\mR} \Sigma(\omega) \rd \omega}}\\
\nonumber
    & =
    \sqrt{
    \frac{1}{2\pi m}
    \int_{\mR} \Sigma(\omega) \Tr \Lambda(\omega)\rd \omega}\\
\label{Fgain0}
    & =
    \sqrt{
    \frac{1}{2\pi m}
    \int_{\mT}
    \Tr \Lambda
    \Big(\Omega \tan \frac{\varphi}{2}\Big)
    \rd\varphi},
\end{align}
since the function $\Sigma$ in (\ref{Sigma}) satisfies $\int_{\mR} \Sigma(\omega) \rd \omega = 2\pi$. Here, we have used the following transformation of the frequency:
\begin{equation}
\label{omphi}
    \omega
    =
    \Omega
    \tan
    \frac{\varphi}{2},
    \qquad
    \varphi
    :=
    2 \arctan (\omega T),
\end{equation}
so that
\begin{equation}
\label{dphidom}
    \d_\omega \varphi = \Sigma(\omega)
\end{equation}
(or, equivalently, $\Sigma(\omega)\rd \omega = \rd \varphi$), where  the new integration variable
$\varphi$ takes values in
the interval
\begin{equation}
\label{mT}
    \mT:= (-\pi,\pi),
\end{equation}
which represents the 
unit circle $\{\re^{i\varphi}: \varphi \in \mT\} = \{z\in \mC: |z|=1\}\setminus \{-1\}$ in the complex plane, punctured at $-1$.
The right-hand side of (\ref{Fgain0}) does not depend on $\kappa$ in accordance with the scale invariance (\ref{kappaS}) and is organised as a weighted $\cH_2$-norm \cite{F_1987} of the transfer function $F$, which involves the parameter $\Omega$ from (\ref{Omega}). We will now discuss the asymptotic behaviour of the RMS gain (\ref{Fgain0}) as a function of $T$.

As $T\to +\infty$ (so that $\Omega \to 0+$), the OU process $W_T$ acquires long-range correlations.
 In this case, the spectral density $\Sigma(\omega)$ in (\ref{Sigma}) converges to $2\pi \delta(\omega)$ in the distributional sense.
The effect of such a process on the system is, in essence, equivalent to that of a constant input.
The corresponding limit
\begin{align}
\nonumber
    \lim_{T\to +\infty}
    \sn F \sn_{I_m}
    & =
    \lim_{\Omega\to 0+}
    \sn F \sn_{I_m}\\
\label{Tinf}
    & =
    \sqrt{
    \frac{1}{m}
    \Tr \Lambda(0)
    }
    =
    \frac{1}{\sqrt{m}}
    \|F(0)\|_{\rF},
\end{align}
which is obtained by applying Lebesgue's dominated convergence theorem  to the RMS gain (\ref{Fgain0}), involves the static gain matrix
$
    F(0) = D-CA^{-1}B
$
of the system in view of (\ref{Ftrans}).

By a similar reasoning,
as the correlation time $T$ of the OU input $W_T$ goes to zero (and hence, $\Omega \to +\infty$), the spectral density  (\ref{Sigma}) becomes constant ($\Sigma(\omega) \approx 2T$) over the widening  frequency interval $|\omega| \ll \Omega$ and the RMS gain (\ref{Fgain0}) approaches
\begin{align}
\nonumber
    \lim_{T\to 0+}
    \sn F \sn_{I_m}
    & =
    \lim_{\Omega\to +\infty}
    \sn F \sn_{I_m}\\
\label{T0}
    & =
    \sqrt{
    \frac{1}{m}
    \Tr \lim_{\omega\to \infty}\Lambda(\omega)
    }
    =
    \frac{1}{\sqrt{m}}
    \|D\|_{\rF}.
\end{align}
The limits in (\ref{Tinf}), (\ref{T0}) do not reduce to the standard $\cH_2$-norm
\begin{equation}
\label{H2}
    \|F\|_2 :=
    \sqrt{\frac{1}{2\pi} \int_{\mR} \Tr \Lambda(\omega)\rd \omega}
\end{equation}
which is finite only when the system is strictly proper, that is, if $D=0$.
In the latter case, the limit in (\ref{T0}) vanishes, and   the asymptotic behaviour of the RMS gain is described by 
$$
    \sn F \sn_{I_m}
     =
    \sqrt{
    \frac{T}{\pi m}
    \int_{\mR}
    \frac{\Tr \Lambda(\omega) }{1 +(\omega T)^2}
    \rd \omega}
    \sim
    \|F\|_2
    \sqrt{\frac{2T}{m}},
$$
as     $T \to 0+$,
where use is made of the last equality in (\ref{Sigma}) in combination with
$    \lim_{T\to 0}\int_{\mR}
    \frac{\Tr \Lambda(\omega) }{1 +(\omega T)^2}
    \rd \omega = 2\pi \|F\|_2^2$ which follows from (\ref{H2}) by Lebesgue's  dominated convergence theorem.

The limits (\ref{Tinf}), (\ref{T0}) in the isotropic white-noise case being considered  manifest themselves when the time scale parameter $T$ of the low-pass filtering is large, or, respectively, small, in comparison with the transient times in the system $F$, which can be described in terms of the  set
$
    \cT
    :=
    \big\{
        \tfrac{1}{|\lambda|}:\
        \lambda \in \fS
    \big\}
$,
where $\fS$ denotes the spectrum of the Hurwitz matrix $A$.
The fulfillment of either of the relations
\begin{equation}
\label{Tggll}
    T\gg \max \cT = \rho(A^{-1}),
    \qquad
    T\ll \min \cT = \frac{1}{\rho(A)}
\end{equation}
(with $\rho(\cdot)$ denoting the spectral radius of a square matrix)
indicates whether the OU process $W_T$ is strongly coloured or nearly white for the system.
The relative simplicity of this comparison comes from the fact that such a process has only one time scale parameter $T$ (which specifies the characteristic width $\Omega$ of the effective frequency range for the OU process), whereas more complicated random inputs can have multiple time scales, similarly to the system $F$ itself.

The behaviour of the RMS gain with respect to isotropic white-noise inputs, considered above, is qualitatively different from its discrete-time counterpart \cite{V_1995a}  despite similarities in their definitions.
Indeed, in the discrete-time settings, the frequency range is finite and can be identified with the interval $\mT$ in (\ref{mT}),
whereas the continuous time processes may contain arbitrarily fast components (with arbitrarily short time scales or, equivalently, high frequencies) and have the infinite frequency range. At the same time, the last representation in (\ref{Fgain0}) is in terms of the $\cH_2$-norm of  a discrete-time transfer function related to 
(\ref{Ftrans}) by a conformal correspondence. This correspondence  applies not only to the isotropic white-noise case but also to a wide class of stationary Gaussian disturbances.

\section{Conformal correspondence between continuous and discrete time settings}
\label{sec:conf}

Consider an involutive conformal map (a modified Cayley transform \cite{S_1992})
\begin{equation}
\label{K}
    K(z) = \frac{1-z}{1+z}
\end{equation}
between the open unit disk $\{z \in \mC:\ |z|<1\}$ and the open right half-plane  $\{s \in \mC:\ \Re s >0\}$.
 These two domains in $\mC$ pertain to
the discrete and continuous time settings, respectively. The map $K$ is a smooth bijection of the punctured unit circle $\{z \in \mC:\ |z|=1\}\setminus \{-1\}$  onto the imaginary axis:
\begin{equation}
\label{MT}
    K(\re^{i\varphi})
    =
    \frac{\re^{-i\varphi/2}-\re^{i\varphi/2}}{\re^{-i\varphi/2}+\re^{i\varphi/2}}
    =
    -i
    \tan \frac{\varphi}{2},
    \qquad
    \varphi \in \mT.
\end{equation}
This property is inherited by the scaled version $\Omega K$ of the map for any $\Omega>0$.
The conformal map $\Omega K$ specifies a linear operator $\Psi$ which maps the transfer functions $F$, $G$ of the system and the shaping filter  in (\ref{Ftrans})--(\ref{Gabcd}) to the transfer functions
\begin{align}
\label{FT}
    F_T(z) & := F(\Omega K(z)),\\
\label{GT}
    G_T(z) & := G(\Omega K(z)),
\end{align}
which are analytic in the open unit disk $|z|<1$ and correspond to stable LDTI systems. Up to a factor of two in $\Omega K$, the operator $\Psi$ is identical to Tustin's method \cite{B_2009} of converting LCTI systems to LDTI systems and vice versa. 
Similarly to (\ref{FF}), (\ref{GG}), we associate with $F_T$, $G_T$ the functions $\Lambda_T, S_T: \mT \to \mH_m^+$  by
\begin{align}
\label{FFT}
    \Lambda_T(\varphi)
    & :=
    F_T(\re^{i\varphi})^*F_T(\re^{i\varphi})
    =
    \Lambda
    \big(-\Omega \tan \frac{\varphi}{2}\big),\\
\label{GGT}
    S_T(\varphi)
    & :=
    G_T(\re^{i\varphi})G_T(\re^{i\varphi})^*
    =
    S\big(-\Omega \tan \frac{\varphi}{2}\big),
\end{align}
where the rightmost equalities follow from (\ref{MT}). Their significance is clarified below.

\begin{theorem}
\label{th:RMST}
Suppose the system (\ref{dZ}), (\ref{dX}) and the shaping filter (\ref{dxi}) are stable. Then the RMS gain (\ref{limFgain}) is representable in terms of the functions (\ref{FFT}), (\ref{GGT}) as
\begin{align}
\label{RMST}
    \sn F\sn_S
    & =
    \sqrt{
    \frac{\int_{\mT} \bra \Lambda_T(\varphi),  S_T(\varphi)\ket \rd \varphi}
    {\int_{\mT} \Tr S_T(\varphi)\rd \varphi}}.
\end{align}
\end{theorem}
\begin{proof}
Similarly to (\ref{Lamom}), the spectral density $S$ in  (\ref{GG}) satisfies
 $   S(-\omega)
    =
    \overline{S(\omega)}
    =
    S(\omega)^\rT$,
and hence, both $\bra \Lambda(\omega), S(\omega)\ket = \bra \Lambda(-\omega), S(-\omega)\ket $ and $\Tr S(\omega) = \Tr S(-\omega)$ are even functions of $\omega\in \mR$. In combination with the frequency transformation in (\ref{omphi}), (\ref{dphidom}), this allows the numerator and denominator in (\ref{limFgain}) to be represented as
\begin{align}
\nonumber
    \int_{\mR}
    \Sigma(\omega)
    &\bra
        \Lambda(\omega),
        S(\omega)
    \ket \rd \omega
    =
    \int_{\mR}
    \Sigma(\omega)
    \bra
        \Lambda(-\omega),
        S(-\omega)
    \ket
    \rd \omega\\
\nonumber
    &=
    \int_{\mT}
    \bra
        \Lambda\big(-\Omega \tan \tfrac{\varphi}{2}\big),
        S\big(-\Omega \tan \tfrac{\varphi}{2}\big)
    \ket
    \rd \varphi\\
\label{num}
    &=
    \int_{\mT}
    \bra
        \Lambda_T(\varphi),
        S_T(\varphi)
    \ket
    \rd \varphi,\\
\nonumber
    \int_{\mR} \Sigma(\omega) &\Tr S(\omega)\rd \omega
    =
    \int_{\mR} \Sigma(\omega) \Tr S(-\omega)\rd \omega\\
\nonumber
    & = \int_{\mT}
    \Tr
    S\big(-\Omega \tan \tfrac{\varphi}{2}\big)
    \rd \varphi\\
\label{den}
    & = \int_{\mT} \Tr S_T(\varphi)\rd \varphi,
\end{align}
where  (\ref{FFT}), (\ref{GGT}) are also used. Substitution of (\ref{num}), (\ref{den}) into (\ref{limFgain}) establishes (\ref{RMST}).
\end{proof}

The above proof 
employs the connection between the spectral density $\Sigma$ of the OU process in (\ref{Sigma}) and the logarithmic derivative $    (\ln K(s))' = \frac{2}{s^2-1}
$ of the conformal map $K$ in (\ref{K}).\footnote{recall that, being an involution, $K$ coincides with its functional inverse $K^{-1}$}

The function $S_T$ in (\ref{GGT}) is the spectral density of an auxiliary stationary zero-mean Gaussian random sequence $\varpi:= (\varpi_k)_{k \in \mZ}$ in $\mR^m$ whose elements are indexed by the set of integers $\mZ$. Their variance (which is constant over time due to stationarity) is
\begin{equation}
\label{varw}
    \bE (|\varpi_k|^2)
    =
    \frac{1}{2\pi}
    \int_{\mT}
    \Tr S_T(\varphi)
    \rd \varphi,
    \qquad
    k \in \mZ.
\end{equation}
The sequence $\varpi$ is the output of an LDTI shaping filter (with the transfer function $G_T$ in (\ref{GT})) driven by a Gaussian white noise sequence $\ups:= (\ups_k)_{k \in \mZ}$ in $\mR^m$ with zero mean and the identity covariance matrix:
\begin{equation}
\label{EvEvv}
    \bE \ups_k = 0,
    \qquad
    \bE (\ups_j\ups_k^\rT) = \delta_{jk} I_m,
    \qquad
    j,k\in \mZ.
\end{equation}
At the same time, $\varpi$ is the input  to an LDTI system (with the transfer function $F_T$ in (\ref{FT})) whose output is a stationary zero-mean Gaussian random sequence $\zeta:= (\zeta_k)_{k \in \mZ}$ in $\mR^p$. The resulting setup is shown in Fig.~\ref{fig:three}.
\begin{figure}[htbp]
\unitlength=0.6mm
\linethickness{0.4pt}
\centering
\begin{picture}(110,75.00)
    \put(35,30){\framebox(10,10)[cc]{{$F$}}}
    \put(50,45){\framebox(10,10)[cc]{{$\Phi$}}}
    \put(80,45){\framebox(10,10)[cc]{{$\Phi$}}}

    \put(39.7,12){\framebox(0.6,17.75)[cc]{}}
    \put(40,10){\makebox(0,0)[cb]{$\blacktriangledown$}}
    \put(69.7,12){\framebox(0.6,17.75)[cc]{}}
    \put(70,10){\makebox(0,0)[cb]{$\blacktriangledown$}}

    \put(35,20){\makebox(0,0)[cc]{$\Psi$}}
    \put(75,20){\makebox(0,0)[cc]{$\Psi$}}

    \put(25,65){\vector(0,-1){10}}
    \put(55,65){\vector(0,-1){10}}
    \put(85,65){\vector(0,-1){10}}
    \put(55,45){\line(0,-1){10}}
    \put(85,45){\line(0,-1){10}}
    \put(25,45){\line(0,-1){10}}
    \put(65,35){\vector(-1,0){20}}
    \put(65,5){\vector(-1,0){20}}
    \put(95,35){\vector(-1,0){20}}
    \put(95,5){\vector(-1,0){20}}
    \put(35,35){\vector(-1,0){20}}
    \put(35,5){\vector(-1,0){20}}
    \put(20,45){\framebox(10,10)[cc]{{$\Phi$}}}
    \put(35,60){\framebox(10,10)[cc]{{$F$}}}
    \put(35,0){\framebox(10,10)[cc]{{$F_T$}}}
    \put(65,60){\framebox(10,10)[cc]{{$G$}}}
    \put(65,30){\framebox(10,10)[cc]{{$G$}}}
    \put(65,0){\framebox(10,10)[cc]{{$G_T$}}}
    \put(35,65){\vector(-1,0){20}}
    \put(65,65){\vector(-1,0){20}}
    \put(95,65){\vector(-1,0){20}}
    \put(10,65){\makebox(0,0)[cc]{$Z$}}
    \put(10,35){\makebox(0,0)[cc]{$Z_T$}}
    \put(10,5){\makebox(0,0)[cc]{$\zeta$}}
    \put(100,65){\makebox(0,0)[cc]{$V$}}
    \put(100,35){\makebox(0,0)[cc]{$V_T$}}
    \put(100,5){\makebox(0,0)[cc]{$\ups$}}

    \put(55,70){\makebox(0,0)[cc]{$W$}}
    \put(55,30){\makebox(0,0)[cc]{$W_T$}}
    \put(55,0){\makebox(0,0)[cc]{$\varpi$}}
\end{picture}
\caption{The block diagram of Fig.~\ref{fig:FGT} augmented by the linear operator $\Psi$ (represented by double arrows) which maps the continuous-time transfer functions $F$, $G$ to their discrete-time counterparts $F_T$, $G_T$ in (\ref{FT}), (\ref{GT}). Also shown are the auxiliary random sequences $\zeta$, $\varpi$, $\ups$ related by the LDTI systems $F_T$, $G_T$. The RMS gain of $F$ with respect to $W_T$ in (\ref{limFgain}) is equal to that of $F_T$ with respect to $\varpi$ in (\ref{RMSconf}).}
\label{fig:three}
\end{figure}
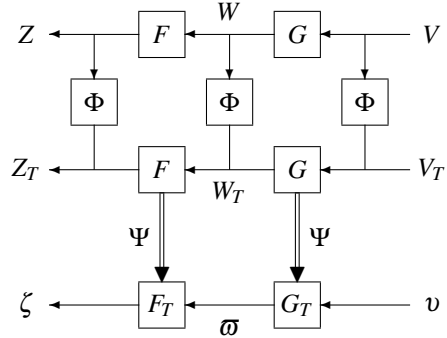
Since the spectral density of $\zeta$ is
$
    F_T(\re^{i\varphi}) S_T(\varphi) F_T(\re^{i\varphi})^*
$, the common variance of its elements is
\begin{align}
\nonumber
    \bE (|\zeta_k|^2)
    & =
    \frac{1}{2\pi}
    \int_{\mT}
    \Tr (F_T(\re^{i\varphi}) S_T(\varphi) F_T(\re^{i\varphi})^*)
    \rd \varphi\\
\label{varz}
    & =
    \frac{1}{2\pi}
    \int_{\mT}
    \bra \Lambda_T(\varphi), S_T(\varphi)\ket
    \rd \varphi,
    \qquad
    k \in \mZ.
\end{align}
In view of (\ref{varw}), (\ref{varz}), the RMS gain (\ref{RMST}) of the system $F$ is identical to that of its discrete-time counterpart $F_T$, specified above:
\begin{equation}
\label{RMSconf}
    \sn F\sn_S
    =
    \sqrt{\frac{\bE (|\zeta_0|^2)}{\bE (|\varpi_0|^2)}}
    =
    \sn F_T\sn_{S_T},
\end{equation}
where, without loss of generality, the variances in the numerator and denominator are taken at the initial moment of time due to stationarity.\footnote{in the stationary case, it is redundant to consider Cesaro means of moments such as variances instead of their constant values}
The right-hand side of (\ref{RMSconf})\footnote{which is the RMS gain of the discrete-time transfer function $F_T$ with respect to stationary zero-mean Gaussian random sequences $\varpi$ with the spectral density $S_T$} depends on the spectral density $S$ of the continuous time process $W$ through its discrete-time image $S_T$ in (\ref{GGT}) and can be computed for a given $S$ both in the continuous and discrete time domains due to the $T$-dependent conformal correspondence.

For example, if $W$ is a standard Wiener process up to a multiplicative constant (the isotropic white noise case (\ref{S0})),  then the spectral density $S_T$ in (\ref{GGT}) is the same constant scalar matrix, and $\varpi$ is a zero-mean Gaussian white noise sequence with the scalar covariance matrix. In this case, (\ref{RMSconf}) reproduces (\ref{Fgain0}).

\section{A continuous-time  extension of the anisotropic norm}\label{sec:an}

If the system input $W$ is a statistically uncertain random disturbance whose
spectral density $S$ is known only up to its membership in a class $\cS$, 
then of interest is the worst-case value
\begin{equation}
\label{supRMS}
    \sn F\sn_{\cS}
    :=
    \sup_{S\in \cS}
    \sn F\sn_S
\end{equation}
of the RMS gain (\ref{RMSconf})
(we have slightly abused the notation). In particular, the class $\cS$ can be described in terms of the deviation of $S$ from constant scalar matrices (which correspond to the isotropic white noise case (\ref{S0})).  Such deviation is quantified, for example, by the mean anisotropy \cite{V_1995a} of the stationary Gaussian random sequence $\varpi$ in $\mR^m$, which, in view of (\ref{FFT}), (\ref{varw}),  takes the form
\begin{align}
\nonumber
    &\bA_T(S)
    :=
    \bA(S_T)\\
\nonumber
    & =
    -\frac{1}{2}
    \ln\det
    \Big(
        \frac{m }{\bE(|\varpi_0|^2)}
        \Gamma
    \Big)\\
\nonumber
    & =
    -\frac{1}{4\pi}
    \int_{\mT}
    \ln\det
    \Big(
        \frac{m}{\bE(|\varpi_0|^2)} S_T(\varphi)
    \Big)
    \rd \varphi\\
\nonumber
    & =
    -
    \frac{1}{4\pi}
    \int_{\mR}
    \Sigma(\omega)
    \ln\det
    \Big(
    \frac{2\pi m }
    {\int_{\mR} \Sigma(\lambda)\Tr S(\lambda)\rd \lambda}
    S(\omega)
    \Big)
    \rd \omega\\
\label{AST}
    & =
    -
    \frac{T}{2\pi}
    \int_{\mR}
    \frac{1}{1+(\omega T)^2}
    \ln\det
    \left(
    \frac{\pi m }
    {T\int_{\mR} \frac{\Tr S(\lambda)}{1+(\lambda T)^2}\rd \lambda}
    S(\omega)
    \right)
    \rd \omega
\end{align}
(the matrix $\Gamma$ is defined below).
This quantity is finite and nonnegative if the spectral density $S$ satisfies the condition
\begin{equation}
\label{PR}
    \int_{\mT}
    \ln\det S_T(\varphi)
    \rd \varphi
    =
    2T
    \int_{\mR}
    \frac{\ln\det S(\omega)}{1 + (\omega T)^2}
    \rd \omega
    >-\infty,
\end{equation}
whose fulfillment does not depend on a particular choice of $T>0$. Since 
the transfer function $G_T$ in (\ref{GT}) belongs to the Hardy space 
of $\mC^{m\x m}$-valued functions, analytic in the open unit disk 
and square integrable over the unit circle, 
then
(\ref{PR}) is equivalent to $\det G_T(\varphi)\ne 0$ for almost all $\varphi \in \mT$. This full-rank condition, associated with the inner-outer factorization \cite{W_1972} of the spectral density $S_T$ in (\ref{GGT}), reflects the absence of linear dependencies between the entries of the random vectors $\varpi_k$ at the same or different moments of time. Also, this property is  closely related to the total unpredictability \cite{R_1967} of such random sequences 
in the sense of nonsingularity of the conditional covariance matrix
\begin{equation}
\label{Gamma}
    \Gamma:= \cov(\varpi_k \mid \cF_{k-1})
\end{equation}
(which does not depend on $k$ and is nonrandom in the stationary Gaussian case being considered)  given the past history
$(\varpi_j)_{j<k}$ of the sequence $\varpi$. Here, the $\sigma$-algebras  $\cF_k := \sigma\{\varpi_k, \varpi_{k-1}, \varpi_{k-2}, \ldots\}$ form the natural filtration for $\varpi$. The connection between $\Gamma$ and $S_T$ is described by the Szego-Kolmogorov theorem \footnote{the use of this relation in \cite[Eq. (3.7) on p. 62]{BAK_2018} contains an error, where the rightmost factor $\frac{1}{4\pi}$ is incorrect and should be replaced with $\frac{1}{2}$ in accordance with the second equality in (\ref{AST})}
\begin{align*}
\nonumber
    \ln \det \Gamma
    & =
    \lim_{N\to+\infty}
    \Big(
    \frac{1}{N}
    \ln \det \cov(\varpi_{0:N-1})
    \Big)\\
    & =
    \frac{1}{2\pi}
    \int_{\mT}
    \ln\det S_T(\varphi)
    \rd \varphi ,
\end{align*}
where $\varpi_{0:N-1}:= (\varpi_k)_{0\< k < N}$ is an $\mR^{mN}$-valued zero-mean Gaussian random vector whose covariance matrix is recovered from the spectral density $S_T$ as
$$    \bE(\varpi_j \varpi_k^\rT)
    =
    \frac{1}{2\pi}
    \int_{\mT}
    \re^{i(j-k)\varphi}
    S_T(\varphi)
    \rd \varphi.
$$
In application to the sequence $\varpi$, the origin of the mean anisotropy functional (\ref{AST}) is clarified by the limit theorems \cite{V_1995a,V_2006}
\begin{align*}
    \lim_{N\to +\infty}
    \Big(&
        \frac{1}{N}
        \bD(\mho_N \| U_{mN})
    \Big)\\
    & =
    \lim_{N\to +\infty}
    \Big(
        \frac{1}{N}
        \min_{\lambda >0}
        \bD(
            \varpi_{0:N-1}
            \|
            \cN_{\lambda I_{mN}}
        )
    \Big)\\
    & =
    \bA(S_T).
\end{align*}
Here, $\bD(\mho_N \| U_{mN})$ is the Kullback-Leibler relative entropy \cite{CT_2006} of the probability distribution of the random unit vector $\mho_N:= \frac{1}{|\varpi_{0:N-1}|} \varpi_{0:N-1}$ with respect to the uniform distribution $U_{mN}$ on the unit sphere in $\mR^{mN}$, and $\bD(\varpi_{0:N-1} \| \cN_{\lambda I_{mN}})$ is the relative entropy of the probability distribution of 
$\varpi_{0:N-1}$ with respect to the Gaussian distribution $\cN_{\lambda I_{mN}}$ in  $\mR^{mN}$ with zero mean and scalar covariance matrix $\lambda I_{mN}$.

Being always nonnegative, the mean anisotropy $\bA_T(S)$ in (\ref{AST}) vanishes if and only if the spectral density $S_T$, and hence, $S$  in (\ref{GGT}), is a constant scalar matrix, which holds only
when the input disturbance $W$ is a standard Wiener process up to a multiplicative constant.
Similarly to discrete-time settings, this allows $\bA_T(S)$ to be used as a measure of deviation from the nominal isotropic white-noise model for specifying the uncertainty class $\cS$ in (\ref{supRMS}):
\begin{equation}
\label{cSTa}
    \cS_{T,a}
    :=
    \{S: \mR\to \mC^{m\x m}:\ \bA_T(S) \< a\}.
\end{equation}
This class consists of all those spectral densities $S$  of the input disturbance process $W$, whose discrete-time counterparts $S_T$ in (\ref{GGT}) satisfy the upper constraint on the mean anisotropy (\ref{AST}) of the corresponding stationary Gaussian sequence $\varpi$ with variance (\ref{varw}). Since the map $S\mapsto S_T$ is bijective, the worst-case RMS gain (\ref{supRMS}), associated with (\ref{cSTa}), takes the form
\begin{align}
\nonumber
    \sn F\sn_{T,a}
    & :=
    \sup_{S \in \cS_{T,a}}
    \sn F \sn_S\\
\label{FTa}
    & =
    \sup_{S \in \cS_{T,a}}
    \sqrt{
    \frac{\int_{\mT} \bra \Lambda_T(\varphi),  S_T(\varphi)\ket\rd \varphi}
    {\int_{\mT} \Tr S_T(\varphi)\rd \varphi}}
    =
        \sn F_T\sn_a
\end{align}
and coincides with the  $a$-anisotropic norm of the discrete-time system $F_T$. The latter norm (and hence, the two-parameter $(T,a)$-anisotropic norm $\sn F\sn_{T,a}$ of the underlying LCTI system)  lends itself to state-space computation.

\section{Computing the continuous-time anisotropic norm in state space}\label{sec:state}

In view of (\ref{Omega}), (\ref{Ftrans}), (\ref{K}), the transfer function $F_T$  in  (\ref{FT}) takes the form
\begin{align}
\nonumber
    F_T(z) &= C\Big(\Omega\frac{1-z}{1+z}I_n - A\Big)^{-1} B + D\\
\nonumber
    & = (1+z) TC(I_n-TA - z(I_n + TA))^{-1}B+D\\
\nonumber
    & = (1+z) C(I_n - z A_T)^{-1}B_T+D\\
\label{FT1}
    & =
    zC_T (I_n - zA_T)^{-1}B_T+D_T
\end{align}
(that is, the generating function for the impulse response with the $z$-transform $F_T(1/z)$)  and corresponds to the LDTI system
\begin{equation}
\label{nextx}
    x_{k+1} = A_T x_k + B_T \varpi_k,
    \qquad
    \zeta_k = C_T x_k + D_T \varpi_k
\end{equation}
with an $\mR^n$-valued state sequence $x:=(x_k)_{k \in \mZ}$ and the state-space realization matrices
\begin{align}
\label{AT}
    A_T
    & := (I_n + TA)(I_n - TA)^{-1},\\
\label{BT}
    B_T
     & :=
    T(I_n - TA)^{-1}B,\\
\label{CT}
    C_T
    & :=
    2C(I_n-TA)^{-1},\\
\label{DT}
    D_T
    & :=
    TC(I_n - TA)^{-1}B + D.
\end{align}
The system (\ref{nextx}) corresponds formally to a finite-difference scheme for numerical integration of the SDEs (\ref{dZ}), (\ref{dX}) by the trapezoidal rule with stepsize $2T$
(such interpretation assumes smallness of the time scale parameter $T$ in the sense of the second relation in (\ref{Tggll})).
For a given $T>0$, the relations (\ref{AT})--(\ref{DT}) describe a smooth bijection $\Ups: (A,B,C,D)\mapsto (A_T, B_T, C_T, D_T)$ between the open sets of matrix quadruples such that $A$ is Hurwitz and $\rho(A_T)<1$. The structure of dependencies between the matrices under the map $\Ups$ is shown in Fig.~\ref{fig:Tustin}
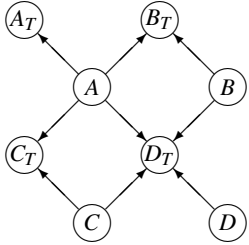
\begin{figure}[htbp]
\unitlength=0.6mm
\linethickness{0.4pt}
\centering
\begin{picture}(110,65.00)
\put(30,60){\makebox(0,0)[cc]{\small$A_T$}}
\put(60,60){\makebox(0,0)[cc]{\small$B_T$}}
\put(30,30){\makebox(0,0)[cc]{\small$C_T$}}
\put(60,30){\makebox(0,0)[cc]{\small$D_T$}}

\put(42,48){\vector(-1,1){9}}
\put(48,48){\vector(1,1){9}}
\put(42,42){\vector(-1,-1){9}}
\put(48,42){\vector(1,-1){9}}
\put(72,48){\vector(-1,1){9}}
\put(42,18){\vector(-1,1){9}}
\put(48,18){\vector(1,1){9}}
\put(72,42){\vector(-1,-1){9}}
\put(72,18){\vector(-1,1){9}}

\put(45,45){\makebox(0,0)[cc]{\small$A$}}
\put(75,45){\makebox(0,0)[cc]{\small$B$}}
\put(45,15){\makebox(0,0)[cc]{\small$C$}}
\put(75,15){\makebox(0,0)[cc]{\small$D$}}

\put(60,30){\circle{8}}

\put(45,45){\circle{8}}
\put(45,15){\circle{8}}
\put(75,45){\circle{8}}
\put(75,15){\circle{8}}

\put(30,60){\circle{8}}
\put(30,30){\circle{8}}
\put(60,60){\circle{8}}
\end{picture}\vskip-8mm
\caption{The graph of dependencies between the matrices under the map $\Ups$, described by (\ref{AT})--(\ref{DT}). Here, $\bigcirc\!\!\!\!\!\!\!u\! \longrightarrow\!\! \! \bigcirc\!\!\!\!\!\!\!v$ indicates dependence of $v$ on $u$. For example, the matrix $A_T$ depends only on $A$, the matrix $B_T$ depends on $A$ and $B$, while the matrix $D_T$ depends on all four matrices $A,B,C,D$.}
\label{fig:Tustin}
\end{figure}
and is inherited by the inverse map $\Ups^{-1}: (A_T, B_T, C_T, D_T)\mapsto (A,B,C,D)$, which, in view of (\ref{Omega}), takes the form
\begin{align}
\label{A}
    A
    & =
    \Omega (I_n+A_T)^{-1}(A_T-I_n),\\
\label{B}
    B
     & =
     2\Omega(I_n+A_T)^{-1}B_T,\\
\label{C}
    C
    & =
    C_T(I_n+A_T)^{-1},\\
\label{D}
    D
    & =
    D_T-C_T(I_n+A_T)^{-1}B_T.
\end{align}
Relations,   similar to (\ref{FT1}), (\ref{AT})--(\ref{D}),  also hold
 for the transfer function $G_T$ of the discrete-time shaping filter in (\ref{GT}). 
As in the discrete-time case, the computation of the $(T,a)$-anisotropic norm  (\ref{FTa}) for the LCTI system being considered  is of interest only if this system is nonround and the mean anisotropy level $a$ is strictly positive (otherwise, the norm reduces trivially to the weighted $\cH_2$-norm $\frac{1}{\sqrt{m}}\|F_T\|_2$).

\begin{theorem}
 \label{th:an}
 Suppose the system (\ref{dZ}), (\ref{dX}) is stable and nonround. Then its $(T,a)$-anisotropic norm (\ref{FTa}) (for any given $T>0$, $a> 0$) can be computed as
 \begin{equation}
 \label{FTa1}
   \sn F\sn_{T,a}
   =
   \sqrt{
   \frac{1}{q}
   \Big(
    1 - \frac{m}{\Tr (L_T P_T L_T^\rT + M_T^2)}
   \Big)
   }.
 \end{equation}
 Here, $P_T=P_T^\rT\in \mR^{n\x n}$ is a positive semi-definite matrix which is the controllability Gramian for the pair $(A_T+B_T L_T, B_T M_T)$ and is a unique solution of the algebraic Lyapunov equation (ALE)
 \begin{equation}
 \label{PALE}
        P_T = (A_T+B_T L_T)P_T (A_T+B_T L_T)^\rT + B_T M_T^2 B_T^\rT,
\end{equation}
where the matrices  $A_T, B_T, C_T, D_T$ are given by (\ref{AT})--(\ref{DT}), and the matrices $L_T \in \mR^{m\x n}$,  $M_T=M_T^\rT\in \mR^{m\x m}$ (with $M_T \succ 0 $) are associated with a unique admissible solution $R_T=R_T^\rT\in \mR^{n\x n}$ of the algebraic Riccati equation (ARE)
\begin{align}
\label{Ric1}
    R_T
    & = A_T^\rT R_T A_T + q C_T^\rT C_T + L_T^\rT M_T^{-2} L_T,\\
\label{Ric2}
    L_T
    & := M_T^2 (B_T^\rT R_T A_T + q D_T^\rT C_T),\\
\label{Ric3}
    M_T
    & := (I_m - qD_T^\rT D_T - B_T^\rT R_T B_T)^{-1/2}.
\end{align}
The admissibility of $R_T$ is understood in the sense that  $I_m - qD_T^\rT D_T - B_T^\rT R_T B_T\succ 0$ is satisfied together with the stability condition
\begin{equation}
\label{stab}
    \rho(A_T+B_T L_T)< 1,
\end{equation}
and the parameter $0< q < \frac{1}{\|F\|_\infty^2}$ (on which the matrices $R_T, L_T, M_T, P_T$ also depend) is a unique solution of the equation
\begin{equation}
\label{logdet}
  -\frac{1}{2}
  \ln\det
  \Big(
    \frac{m }{\Tr (L_T P_T L_T^\rT + M_T^2)}
    M_T^2
  \Big)
  =
  a.
\end{equation}
The worst-case spectral density is unique up to a multiplicative positive constant and is implemented by an input disturbance $W$ according to the SDE
\begin{equation}
\label{Wworst}
  \rd W = L X\rd t + M \rd V,
\end{equation}
where the matrices $L\in\mR^{m\x n}$, $M\in\mR^{m\x m}$ (with $M$ not necessarily symmetric) are given by
\begin{align}
\label{L}
    L & = L_T(I_n + A_T + B_TL_T)^{-1},\\
\label{M}
    M & = (I_m-LB_T)M_T.
\end{align}
\end{theorem}
\begin{proof}
In regard to computing the $(T,a)$-anisotropic norm (\ref{FTa}),  the assertion of the theorem follows from the discrete-time result \cite[Theorem~2]{V_1996a} in application to the system (\ref{nextx})--(\ref{DT}). However, we will provide the key ideas of the proof elucidating the probabilistic structure of the worst-case disturbance. Note that the stability and nonroundness of $F$ are equivalent to the corresponding properties  of its discrete-time counterpart $F_T$. Furthermore, both systems have equal $\cH_\infty$-norms in the appropriate Hardy spaces associated with the right half-plane and unit disk: $\|F\|_\infty  = \|F_T\|_\infty$.
Recall that the accompanying material of \cite[Section~5]{V_1996a} (given in more detail in \cite[Section~8]{V_2001a}) describes the structure
\begin{equation}
\label{worst}
    \varpi_k = L_T x_k + M_T \ups_k
\end{equation}
of the input disturbance of the discrete-time system (see Fig.~\ref{fig:worst}),
\begin{figure}[htbp]
\unitlength=0.6mm
\linethickness{0.4pt}
\centering
\begin{picture}(125,40.00)
    \put(35,25){\framebox(10,10)[cc]{{$F_T$}}}
    \put(65,30){\circle{6}}
    \put(65,30){\makebox(0,0)[cc]{\small$+$}}

    \put(85,25){\framebox(10,10)[cc]{{$M_T$}}}
    \put(60,0){\framebox(10,10)[cc]{{$L_T$}}}
    \put(65,10){\vector(0,1){17}}
    \put(40,26){\makebox(0,0)[cc]{\tiny$\bullet$}}
    \put(40,26){\line(0,-1){21}}
    \put(40,5){\vector(1,0){20}}
    \put(35,30){\vector(-1,0){20}}
    \put(62,30){\vector(-1,0){17}}
    \put(85,30){\vector(-1,0){17}}
    \put(115,30){\vector(-1,0){20}}

    \put(10,30){\makebox(0,0)[cc]{$\zeta$}}
    \put(120,30){\makebox(0,0)[cc]{$\ups$}}
    \put(35,5){\makebox(0,0)[cc]{$x$}}

    \put(55,35){\makebox(0,0)[cc]{$\varpi$}}
\end{picture}
\caption{The worst-case shaping filter for the discrete-time system $F_T$ is organised as a ``parasitic'' feedback loop, through which the current state $x_k$ of the system enters the predictable part $L_T x_k$ of the Doob decomposition (\ref{Doob}) of the input disturbance $\varpi_k$.
}
\label{fig:worst}
\end{figure}
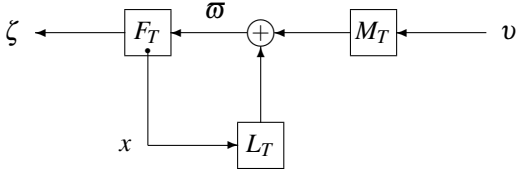
implementing the worst-case spectral density $S_T$ in (\ref{GGT}), which
is unique up to a multiplicative positive constant \cite[Theorem~1]{V_1996a}  and has the form
$S_T = (I_m - q\Lambda_T)^{-1}$ for some $0< q < \frac{1}{\|F\|_\infty^2}$. This representation of $S_T$ in terms of $\Lambda_T$ is equivalent to the isometric property of an auxiliary discrete-time system
\begin{equation}
\label{Theta}
\Theta_T
:=
\begin{bmatrix}
  \sqrt{q} F_T
  \\
  G_T^{-1}
\end{bmatrix}
=
\left[
\begin{array}{c|c}
  A_T & B_T\\
  \hline
  \sqrt{q} C_T & \sqrt{q}D_T\\
  -M_T^{-1}L_T & M_T^{-1}
\end{array}
\right]
\end{equation}
 with the input $\varpi$ and output ${\small\begin{bmatrix}
  \sqrt{q} \zeta\\ \ups
\end{bmatrix}}$.
The right-hand side of (\ref{Theta}) describes the state-space realization matrices of the system $\Theta_T$, with its bottom block row corresponding to the equation $\ups_k = -M_T^{-1}L_T x_k +M_T^{-1}\varpi_k$ obtained from (\ref{worst}) due to $\det M_T \ne 0$.
The condition that the system $\Theta_T$ is isometric is understood in the sense that if $\varpi$ were an arbitrary square summable sequence in  $\mR^m$ (instead of being a stationary Gaussian sequence whose sample paths are square summable only in  the trivial case $\varpi = 0$) and hence, so also were $\zeta$, $\ups$, then their $\ell_2$-norms would be related by $\|\varpi\|_2^2 = q \|\zeta\|_2^2 + \|\ups\|_2^2$. However, in the case of the worst-case stationary  Gaussian sequence $\varpi$ being considered, the isometric property of the system  (\ref{Theta}) manifests itself in terms of the variances as $\bE (|\varpi_k|^2) =
q \bE(|\zeta_k|^2) + m $, where use is also made of the relation $\bE(|\ups_k|^2) = \Tr I_m = m$ which follows from (\ref{EvEvv}). With the matrix $R_T$ in (\ref{Ric1}) being the observability Gramian of the system (\ref{Theta}), the other two equations in the ARE (\ref{Ric1})--(\ref{Ric3}) describe a sufficient state-space condition\footnote{this condition is also necessary if the pair $(A_T,B_T)$ is controllable \cite{GTOP_1989}} for the system $\Theta_T$ to be isometric \cite{V_1996a}.  Similarly to its inverse $G_T^{-1}$ in (\ref{Theta}), the shaping filter
\begin{equation}
\label{GT2}
    G_T
    =
    \left[
    \begin{array}{c|c}
        A_T+B_T L_T  & B_T M_T\\
        \hline
        L_T & M_T
    \end{array}
    \right],
\end{equation}
governed by the first of the equations (\ref{nextx}) in combination with (\ref{worst}),  shares the common state $x$ with the system $F_T$:
\begin{equation}
\label{Gworst}
    x_{k+1} = (A_T+B_T L_T) x_k + B_T M_T \ups_k,
    \quad
    \varpi_k = L_T x_k + M_T \ups_k,
\end{equation}
where, as before, the input of the filter is the Gaussian white-noise sequence $\ups$ in  $\mR^m$ with zero mean and identity covariance matrix. The stability of the filter
(\ref{Gworst}) is secured by the condition (\ref{stab}) and, in combination with  (\ref{EvEvv}), leads to the equation (\ref{PALE}) for the covariance matrix $P_T := \bE(x_k x_k^\rT)$ of the system state under the worst-case input disturbance being considered. Accordingly, the common denominator in (\ref{FTa1}), (\ref{logdet}) is the variance of the output of the shaping filter: $\bE(|\varpi_k|^2) = 
\Tr (L_T P_T L_T^\rT + M_T^2)$ in view of the symmetry $M_T=M_T^\rT$. Due to the stability of the shaping filter and the system itself ($\rho(A_T)< 1$), and also since $\det M_T\ne 0$, the sequence $\varpi$ generates the same filtration $\cF:= (\cF_k)_{k \in \mZ}$ as $\ups$, with (\ref{worst}) providing its Doob decomposition \cite{S_1996} into the $\cF$-predictable and innovation components
\begin{equation}
\label{Doob}
    \bE(\varpi_k \mid \cF_{k-1})
    =
    L_T x_k,
    \qquad
    \varpi_k - \bE(\varpi_k \mid \cF_{k-1}) = M_T \ups_k.
\end{equation}
 Hence, with $M_T$ being symmetric, the conditional covariance matrix (\ref{Gamma}) takes the form $\Gamma = \cov(M_T \ups_k \mid \cF_{k-1}) = M_T \cov(\ups_k) M_T^\rT = M_T^2$, thus leading to the numerator $M_T^2$ in (\ref{logdet}).
The structure of the equivalent continuous-time  input disturbance $W$ is obtained by applying the inverse map $\Ups^{-1}$ (in accordance with (\ref{A})--(\ref{D}))  to the worst-case shaping filter (\ref{GT2}), which leads to (\ref{Wworst})--(\ref{M}). Therefore, in view of (\ref{dX}), the worst-case continuous-time shaping filter $G$ is governed by
\begin{equation}
\label{ABL}
    \rd X
    =
    (A+BL) X \rd t  + BM \rd V,
    \qquad
    \rd W = L X\rd t + M \rd V.
\end{equation}
Due to the properties of the map $\Ups$ (see also Fig.~\ref{fig:Tustin}),  the dynamics matrix of the SDE (\ref{ABL}) satisfies
$
    A + BL
    =
    \Omega (I_n+A_T+B_T L_T)^{-1}(A_T+B_T L_T-I_n)
$ and,  in view of the condition (\ref{stab}), is Hurwitz.
\end{proof}

Similarly to the discrete-time representation in Fig.~\ref{fig:worst}, the worst-case  disturbance $W$ at the input of the LCTI system (\ref{dZ}), (\ref{dX}) is formed in (\ref{Wworst}) by the feedback loop,  through which the system state $X$ enters the drift $LX$ of the disturbance.

\section{An example of computing the two-parameter anisotropic norm}\label{sec:ex}

Consider a stable LCTI system with the state, input and output dimensions $n=4$, $m=3$, $p=2$ and the following state-space realization matrices:
\begin{align*}
&\left[{\begin{array}{c|c}
  A & B \\
  \hline
  C & D
\end{array}}
\right]\\
&= \left[
    {\tiny\begin{array}{cccc|ccc}
   -1.8396 &   0.1240 &  -1.2078 &  -1.0582&    -0.2779 &  -0.8236 &   0.0335\\
    1.3546 &   0.4367 &   2.9080 &  -0.4686&     0.7015 &  -1.5771 &  -1.3337\\
   -1.0722 &  -1.9609 &  -0.1748 &  -0.2725&    -2.0518 &   0.5080 &   1.1275\\
    0.9610 &  -0.1977 &   1.3790 &   0.0984&    -0.3538 &   0.2820 &   0.3502 \\
    \hline
   -0.2991 &  -0.2620 &  -0.2857 &  -0.9792&   -0.5336  &  0.9642  & -0.0200\\
    0.0229 &  -1.7502 &  -0.8314 &  -1.1564&   -2.0026  &  0.5201  & -0.0348
    \end{array}}
    \right].
\end{align*}
The eigenvalues of $A$ are
$-0.2406 \pm 2.0265i$, $-0.7598$, $-0.2384$, and the corresponding bounds in (\ref{Tggll}) are $
    \tfrac{1}{\rho(A)} = 0.4900$,
    $\rho(A^{-1}) = 4.1953
$. The lower of these bounds exceeds noticeably the time scale parameter $T=0.1890$ of the conformal correspondence (with the cutoff frequency $\Omega = 5.2923$ in (\ref{Omega})) used below. Therefore, the effective frequency range of the low-pas filter in (\ref{dVT})--(\ref{dZT}), which specifies the filtered input and output $W_T$, $Z_T$ for the RMS gain (\ref{limFgain}), reflects adequately the transient processes in the system. The matrices
(\ref{AT})--(\ref{DT}) of the discrete-time counterpart of the original system take the form
\begin{align*}
&\left[{\begin{array}{c|c}
  A_T & B_T \\
  \hline
  C_T & D_T
\end{array}}
\right]\\
&= \left[
    {\tiny\begin{array}{cccc|ccc}
    0.5246 &   0.1804 &  -0.3191 &  -0.3102 &   0.0442 &  -0.1691 &  -0.0622 \\
    0.1844 &   0.8411 &   0.8756 &  -0.2496 &  -0.0442 &  -0.2533 &  -0.1464 \\
   -0.3739 &  -0.6849 &   0.6671 &   0.0505 &  -0.3604 &   0.2125 &   0.2644 \\
    0.1758 &  -0.2185 &   0.3503 &   1.0034 &  -0.1540 &   0.0891 &   0.1317 \\
    \hline
   -0.5696 &  -0.1267 &  -0.9532 &  -1.8180 &  -0.2815 &   0.9332 &  -0.1675 \\
   -0.1803 &  -2.3962 &  -3.3308 &  -1.9290 &  -1.4465 &   0.6799 &  -0.1520
    \end{array}}
    \right].
\end{align*}
The behaviour of the singular values of the transfer function $F_T$ in (\ref{FT1}) on the unit circle
(see Fig.~\ref{fig:FTsingval}) show that the system is substantially nonround.
\begin{figure}[htb]
\begin{center}
\includegraphics[scale=0.29]{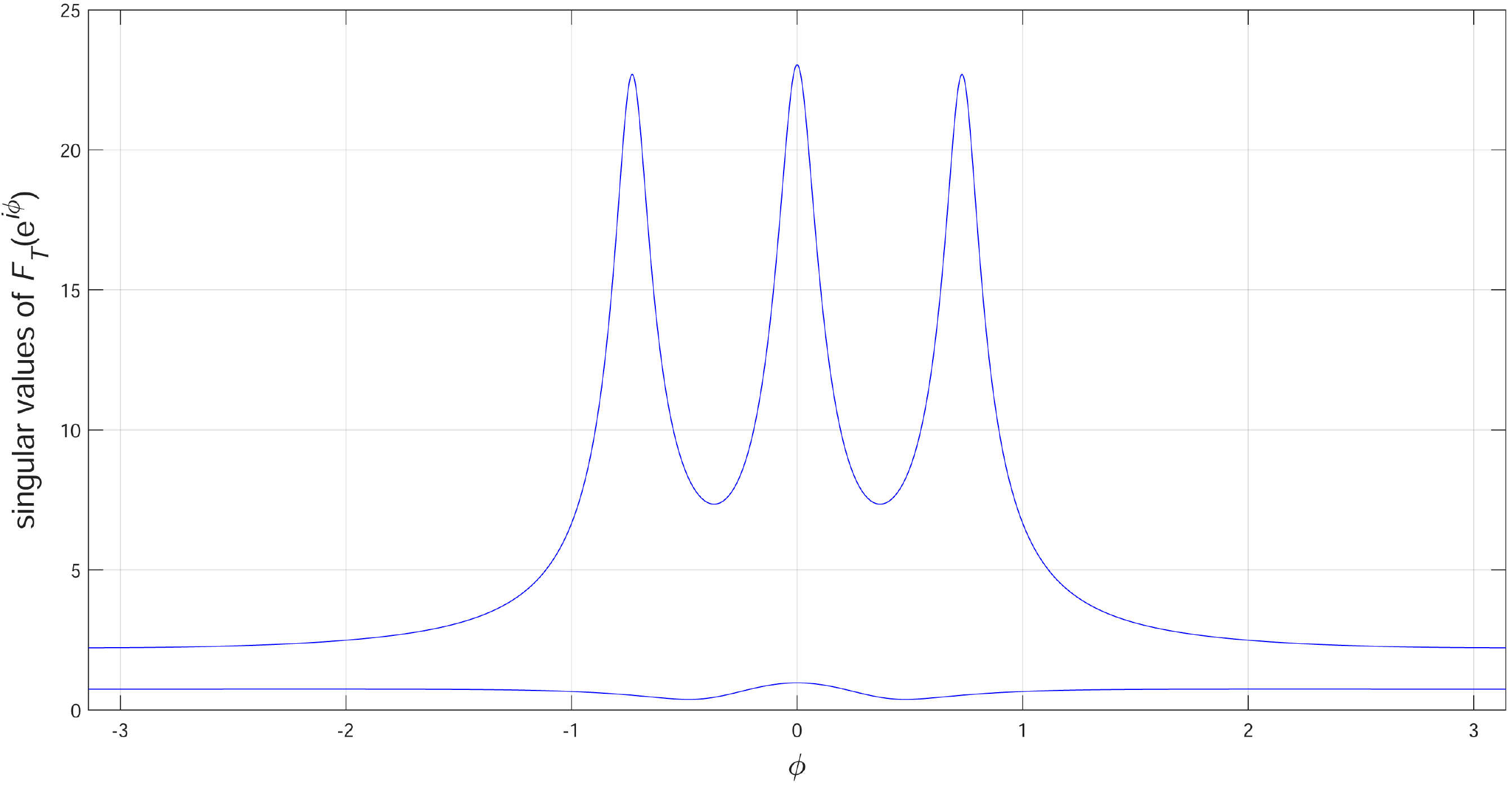}
\caption{Singular values of the discrete-time transfer matrix $F_T(z)$ in (\ref{FT1}) on the unit circle $|z|=1$.}\label{fig:FTsingval}
\end{center}
\end{figure}
Accordingly, there is a large gap between the scaled $\cH_2$-norm and the $\cH_\infty$-norm  of the discrete-time system:
$
    \frac{1}{\sqrt{m}}
    \|F_T\|_2 = 4.6211$,
$
    \|F_T\|_{\infty} = 23.0381
$. These norms are the endpoints of the range of the $(T,a)$-anisotropic norm $\sn F\sn_{T,a}$ of the system (as the mean anisotropy level $a$ varies from $0$ to $+\infty$),  computed using Theorem~\ref{th:an} (with Newton's iterations \cite[Section 4]{V_1996a}, \cite[Section 9]{V_2001a}) and shown in Fig.~\ref{fig:FTa}.
\begin{figure}[htb]
\begin{center}
\includegraphics[scale=0.29]{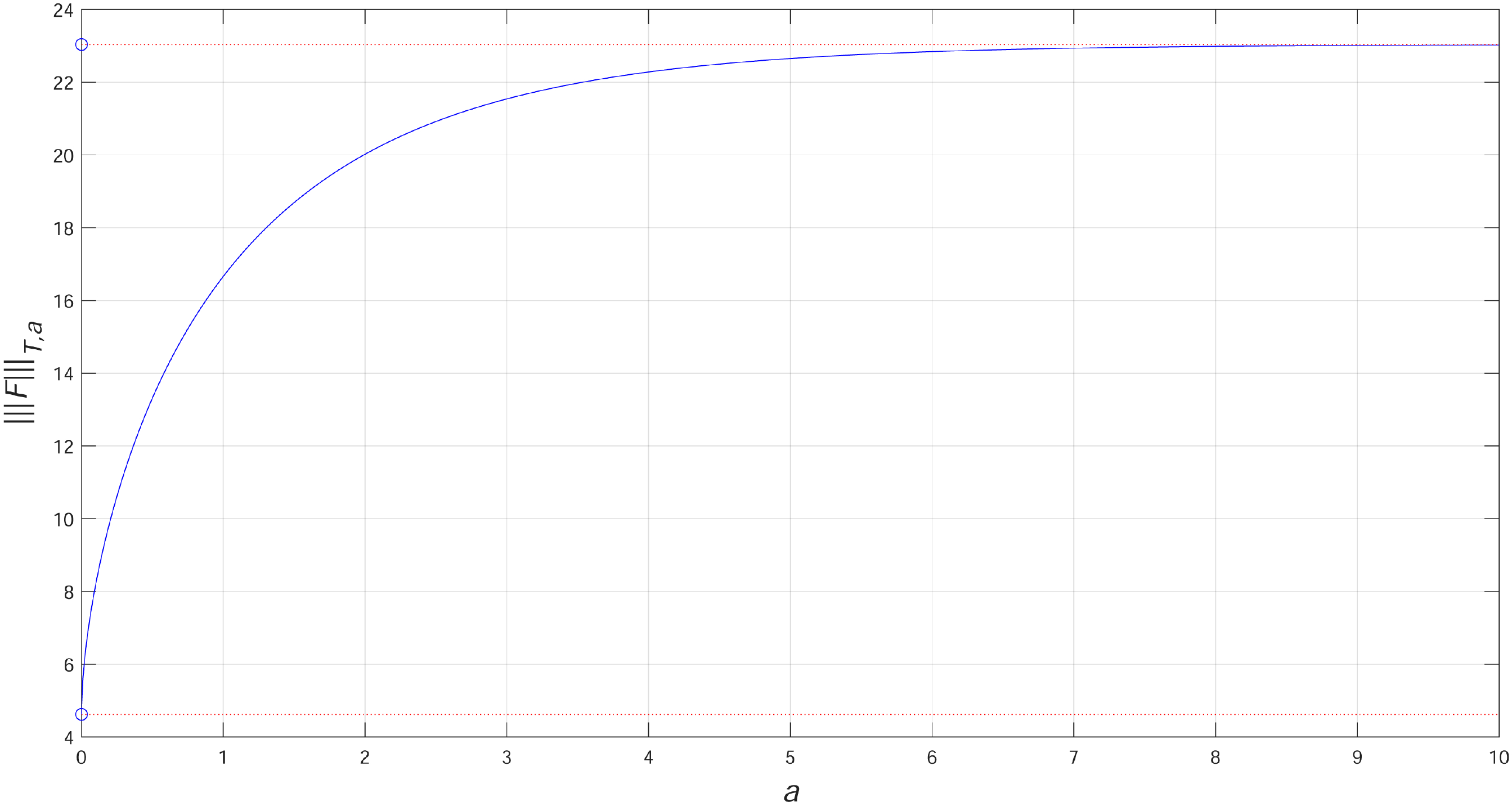}
\caption{The two-parameter anisotropic norm $\sn F\sn_{T,a}$ of the LCTI system as a function of the input mean anisotropy level $a\>0$ for the discrete-time counterpart of the system.}\label{fig:FTa}
\end{center}
\end{figure}
For example, at $a=1.2264$, the norm being considered is $\sn F\sn_{T,a} = 17.6938$, and the matrices (\ref{L}), (\ref{M}) of the worst-case shaping filter (\ref{Wworst}) take the form
\begin{align*}
    L
    & =
    {\small\begin{bmatrix}
   -0.0286 &   0.0285 &  -0.0657 &  -0.0423\\
   -0.0220 &  -0.1366 &  -0.0369 &   0.1185\\
    0.0356 &  -0.1055 &  -0.0105 &   0.1713
    \end{bmatrix}},\\
    M
& =
    {\small\begin{bmatrix}
    1.0040 &   0.0091 &   0.0048\\
   -0.0121 &   1.0011 &   0.0034\\
   -0.0046 &  -0.0036 &   1.0000
    \end{bmatrix}}.
\end{align*}
Note that $M$ is close to the identity matrix, while $L$ is relatively small.

\section{Conclusion}\label{sec:conc}

We have discussed a particular way of extending the discrete-time anisotropy-based criteria
to LCTI systems governed by SDEs with statistically uncertain Ito processes at the input.  This approach uses the RMS gain of the system in terms of filtered versions of the input and output processes. The time scale $T$ of the low-pass filter parameterises Tustin's transform of the original LCTI system $F$ to its LDTI counterpart $F_T$ with the same RMS gain, though understood in the usual sense for fragments (or steady-state RMS values) of stationary Gaussian random sequences.

The resulting two-parameter norm
$\sn F\sn_{T,a}$ is defined as the  $a$-anisotropic norm $\sn F_T\sn_a$ of the effective discrete-time system, with the input mean anisotropy threshold $a$ quantifying indirectly the deviation of the spectral density of the Ito process at the input of the original system from constant scalar matrices. The computation of this norm reduces to the numerical solution of a set of matrix algebraic Riccati, Lyapunov and log-determinant equations 
developed more than 20 years ago.
Under Tustin's inverse transform, the worst-case disturbance  inherits the general feedback structure, whereby its drift depends linearly on the current state of the system.

Therefore, in combination with Tustin's direct and inverse transforms, the methods of  anisotropy-based robust performance analysis and control design, developed earlier for discrete-time systems, are applicable to continuous-time settings.  Alternative ways of extending the anisotropy-based theory to continuous-time systems, employing a sample path analysis of Gaussian diffusion processes under time discretization, will be discussed in subsequent publications.

\section*{Acknowledgements}

The original results of the discrete-time  anisotropy-based theory were obtained by the author under the support of the Russian Foundation for Basic Research grant 95-01-00447 in 1995--1997
(while the author was with the State Research Institute of Aviation Systems and the Institute
for Information Transmission Problems, the Russian Academy of Sciences)  and the
Australian Research Council grant A10027063 in 2000--2002 (while the author was with the Mathematics Department
of the University of Queensland).






\end{document}